\documentclass{amsart}
\usepackage{amssymb}
\usepackage{mathrsfs}
\usepackage[all]{xy}

\newcommand{\C}{\mathbb{C}}
\newcommand{\F}{\mathbb{F}}
\newcommand{\Q}{\mathbb{Q}}
\newcommand{\Z}{\mathbb{Z}}

\newcommand{\bk}{\Bbbk}
\newcommand{\fm}{\mathfrak{m}}

\newcommand{\fg}{\mathfrak{g}}

\newcommand{\cB}{\mathscr{B}}

\newcommand{\fH}{\mathfrak{H}}
\newcommand{\bW}{\mathbf{W}}

\newcommand{\cN}{\mathcal{N}}
\newcommand{\tcN}{\widetilde{\mathcal{N}}}

\newcommand{\cO}{\mathcal{O}}

\newcommand{\tfg}{\widetilde{\mathfrak{g}}}

\newcommand{\Lie}{\mathrm{Lie}}
\newcommand{\rs}{\mathrm{rs}}

\newcommand{\Gr}{\mathsf{Gr}}
\newcommand{\fN}{\mathfrak{N}}
\newcommand{\Gv}{{\check G}}
\newcommand{\cM}{\mathcal{M}}

\newcommand{\allmod}{\text{\textrm{-}}\mathrm{mod}}

\newcommand{\Rep}{\mathsf{Rep}}

\newcommand{\For}{\mathrm{For}}
\newcommand{\uo}{{}^\circ \hspace{-1.5pt}}
\newcommand{\ut}{{}^{\mathrm{t}} \hspace{-1.5pt}}
\newcommand{\Formod}{\mathrm{For}^{\mathrm{m}}}

\newcommand{\BC}{\mathrm{(BC)}}
\newcommand{\Co}{\mathrm{(Co)}}
\newcommand{\Const}{\mathrm{(Const)}}
\newcommand{\Ss}{\mathbb{S}}
\newcommand{\Ex}{\mathrm{E}}

\def\DB{{\mathbf D}}
\def\EB{{\mathbf E}}

\newcommand{\Sh}{\mathsf{Sh}}
\newcommand{\Shc}{\mathsf{Sh}_{\mathrm{c}}}
\newcommand{\Db}{D^{\mathrm{b}}}
\newcommand{\Dbc}{D^{\mathrm{b}}_{\mathrm{c}}}
\newcommand{\Dbctf}{D^{\mathrm{b}}_{\mathrm{ctf}}}
\newcommand{\Dcon}{\Db_{\mathrm{con}}}
\newcommand{\Perv}{\mathsf{Perv}}
\newcommand{\ub}[1]{\underline{#1}}
\newcommand{\ubb}{\ub{(-)}}
\newcommand{\bT}{\mathbb{T}}

\newcommand{\cA}{\mathcal{A}}
\newcommand{\cL}{\mathcal{L}}

\newcommand{\IC}{\mathrm{IC}}

\newcommand{\wt}{\check\omega}

\newcommand{\Spr}{\underline{\mathsf{Spr}}}
\newcommand{\Groth}{\underline{\mathsf{Groth}}}

\newcommand{\drho}{\dot\rho}
\newcommand{\dvarphi}{\dot\varphi}
\newcommand{\dr}{\dot{r}}

\newcommand{\simto}{\xrightarrow{\sim}}

\newcommand{\End}{\mathrm{End}}
\DeclareMathOperator{\Hom}{Hom}

\DeclareMathOperator{\rk}{rk}

\DeclareMathOperator{\Irr}{Irr}
\newcommand{\id}{\mathrm{id}}

\newcommand{\sgn}{\mathsf{sgn}}
\newcommand{\sC}{\mathscr{C}}

\newtheorem*{thm*}{Theorem}
\numberwithin{equation}{section}
\newtheorem{thm}{Theorem}[section]
\newtheorem{lem}[thm]{Lemma}
\newtheorem{prop}[thm]{Proposition}
\newtheorem{cor}[thm]{Corollary}
\theoremstyle{definition}

\theoremstyle{remark}
\newtheorem{rmk}[thm]{Remark}
\newtheorem{ex}[thm]{Example}

\title[Weyl group actions on the Springer sheaf]{Weyl group actions on the Springer sheaf}

\author{Pramod N. Achar}
\address{Department of Mathematics\\
  Louisiana State University\\
  Baton Rouge, LA 70803\\
  U.S.A.}
\email{pramod@math.lsu.edu}

\author{Anthony Henderson}
\address{School of Mathematics and Statistics\\
  University of Sydney, NSW 2006\\
  Australia}
\email{anthony.henderson@sydney.edu.au}

\author{Daniel Juteau}
\address{Laboratoire de Math\'ematiques Nicolas Oresme\\
  Universit\'e de Caen, BP 5186\\
  14032 Caen Cedex\\ 
  France}
\email{daniel.juteau@unicaen.fr}

\author{Simon Riche}
\address{Universit{\'e} Blaise Pascal - Clermont-Ferrand II, Laboratoire de Math{\'e}matiques, CNRS, UMR 6620, Campus universitaire des C{\'e}zeaux, F-63177 Aubi{\`e}re Cedex, France
}
\email{simon.riche@math.univ-bpclermont.fr}


\thanks{P.A. was supported by NSF Grant No.~DMS-1001594.  A.H. was supported by ARC Future Fellowship Grant No.~FT110100504. D.J. was supported by ANR Grant No.~ANR-09-JCJC-0102-01. S.R. was supported by ANR Grants No.~ANR-09-JCJC-0102-01 and ANR-2010-BLAN-110-02.}

\begin{document}

\begin{abstract}
We show that two Weyl group actions on the Springer sheaf with arbitrary coefficients, one defined by Fourier transform and one by restriction, agree up to a twist by the sign character. This generalizes a familiar result from the setting of $\ell$-adic cohomology, making it applicable to modular representation theory. We use the Weyl group actions to define a Springer correspondence in this generality, and identify the zero weight spaces of small representations in terms of this Springer correspondence.
\end{abstract}

\maketitle

\section{Introduction}
\label{sect:intro}

\subsection{}

In the 1970s, T.A.~Springer made the surprising
discovery~\cite{springer} that the Weyl group $\bW$ of a reductive
group $G$ acts in a natural way on the $\ell$-adic cohomology of
Springer fibres.  Shortly thereafter, Lusztig~\cite{lus} 
gave another construction of such an action, using the brand-new
language of perverse sheaves. It is a fact of fundamental importance
that these two constructions almost coincide:
they differ only by the sign character
of $\bW$, as shown in~\cite{hotta}. Thus there is essentially only one
Springer correspondence over a field of characteristic zero. The aim of this paper is to prove
a similar result with arbitrary coefficients.

The study of Springer theory with coefficients in a local ring or a finite field was initiated by the third author in his thesis~\cite{juteau}.  That work, following~\cite{hk, bry}, constructed the $\bW$-action via a kind of Fourier transform (as did Springer, although his construction was 
different).  On the other hand, in the later paper~\cite{ahr}, the other 
authors worked with a $\bW$-action defined as in~~\cite{lus, bm}.  
One purpose of this paper is to fill in the relationship between~\cite{juteau} and~\cite{ahr}, and hence derive further consequences of the results of~\cite{ahr}. The main result of this paper is also used by Mautner in his geometric study of Schur--Weyl duality~\cite{mautner}.

\subsection{}

A 
precise statement is as follows.  Let $G$ be a connected reductive algebraic group over $\C$. Let $\bk$ be a commutative noetherian ring of finite global dimension.  The \emph{Springer sheaf with coefficients in $\bk$}, denoted $\Spr(\bk)$, is the perverse sheaf on the nilpotent cone $\cN$ of $G$ obtained by pushing forward the appropriate shift of the constant sheaf $\underline{\bk}$ 
along the Springer resolution.  (For full definitions,
 see Section~\ref{sect:prelim}.)  There are two actions of $\bW$ on $\Spr(\bk)$, giving rise to 
\begin{align*}
\dvarphi&: \bk[\bW] \to \End(\Spr(\bk)) &&\text{given by `Fourier transform' as in~\cite{bry, juteau},} \\
\drho&: \bk[\bW] \to \End(\Spr(\bk)) &&\text{given by `restriction' as in~\cite{bm, ahr}.}
\end{align*}
It follows from general properties of the Fourier transform that $\dvarphi$ is a ring isomorphism.  Next, let $\varepsilon : \bW \to \{ \pm 1\}$ be the sign character, and let
\[
\sgn : \bk[\bW] \to \bk[\bW]
\qquad\text{be given by}\qquad
\sgn(w) = \varepsilon(w)w\qquad\text{for }w \in \bW.
\]
The main result of the paper is the following.

\begin{thm}\label{thm:main-intro}
We have $\drho = \dvarphi \circ \sgn$.
\end{thm}

An immediate corollary is that $\drho$ is also an isomorphism, generalizing the result for $\bk = \Q$ proved by Borho and MacPherson~\cite[Th\'eor\`eme 3]{bm}.
Further consequences, concerning the Springer correspondence induced by $\dvarphi$ (or equivalently by $\drho$) and its relation to the zero weight spaces of small representations, are obtained in Section~\ref{sect:zeroweight}. 

\subsection{}
\label{ss:strategy}
The `easy' case of Theorem~\ref{thm:main-intro} is when the following condition holds:
$\bk$ embeds 
either in a $\Q$-algebra,
or in an $\F_p$-algebra where $p$ does not divide $|\bW|$.
Under this assumption, 
Theorem~\ref{thm:main-intro} can be proved by arguments that have already appeared in the $\ell$-adic setting; we explain this proof in Section~\ref{sect:thm-Q}. The main idea is a comparison with the classical $\bW$-action on the cohomology of the flag variety $\cB$ of $G$. The condition on $\bk$
is needed to ensure that $H^\bullet(\cB;\bk)$ is a faithful $\bk[\bW]$-module.
Note that $\bk=\Z$ satisfies the condition,
but the case of most interest in modular representation theory, when $\bk$ is a field of characteristic dividing $|\bW|$, does not.

We will deduce the general case from the case $\bk=\Z$. In order to do
so, one might try to use the morphism $\End(\Spr(\Z))
\to \End(\Spr(\bk))$ induced by the `extension of scalars' functor
$(-) \overset{\scriptscriptstyle L}{\otimes}_{\Z} \bk : \Db(\cN,\Z)
\to \Db(\cN,\bk)$. However, to follow this strategy one would have to check
that many isomorphisms between sheaf functors are compatible with
extension of scalars. It turns out to be more convenient to use `restriction of scalars' (or
forgetful) functors $\For_{\Z,\bk}$. One slight drawback is that there is no obvious morphism
$\End(\Spr(\Z)) \to \End(\Spr(\bk))$ induced by
$\For^{\cN}_{\Z,\bk}$. To bypass this difficulty
we regard $\Spr^{\bk}$ as a \emph{functor} from the category of
$\bk$-modules to the category of perverse sheaves on $\cN$, and we
study the ring $\End(\Spr^{\bk})$ of natural endotransformations of
this functor.  The main ideas of Springer theory can
be developed in this functorial setting. A functor is more
`rigid' than a single perverse sheaf, 
which allows us to define the desired morphism $\End(\Spr^{\Z}) \to \End(\Spr^{\bk})$ compatible with the functorial versions of $\dvarphi$ and $\drho$.
That compatibility enables us to deduce Theorem~\ref{thm:main-intro} from the $\bk = \Z$ case. 

\subsection{Outline of the paper}

Section~\ref{sect:prelim} contains the basic constructions of Springer theory in our functorial setting.
Section~\ref{sect:restrict} contains the 
argument that reduces Theorem~\ref{thm:main-intro} to the $\bk=\Z$ case. In Section~\ref{sect:thm-Q}, we explain how to prove Theorem~\ref{thm:main-intro} in the `easy' case,
thus completing the proof. Section~\ref{sect:zeroweight} presents the applications to the Springer correspondence and to small representations. Finally, Section~\ref{sect:etale} discusses an analogue of Theorem~\ref{thm:main-intro} for the \'etale topology over a finite field. Appendix~\ref{sect:isomsheaf} collects some general lemmas expressing the compatibility of various isomorphisms of sheaf functors.

\subsection{Acknowledgements}

We are grateful to Frank L\"ubeck for supplying data about small
representations in positive characteristic that supplemented \cite{LUB}.

\section{Preliminaries}
\label{sect:prelim}

\subsection{Groups and varieties}

Let $G$ be a connected reductive algebraic group over $\C$, 
with Lie algebra $\fg$.  
Let $\cB$ be the flag variety, which parametrizes Borel subgroups of $G$.  If $B \in \cB$, we let
$U_B$ be its unipotent radical.  Let $\fH$ be the universal Cartan subalgebra of $\fg$. Recall that if $B$ is a Borel subgroup of $G$, $\fH$ is canonically isomorphic to $\Lie(B)/\Lie(U_B)$.  Let $\bW$ be the universal Weyl group; 
by definition, it acts naturally on $\fH$.

Consider the vector bundle 
\[
\tfg = \{(x,B) \in \fg \times \cB \mid x \in \Lie(B) \}
\]
over $\cB$. We have a canonical morphism
\[
\pi: \tfg \to \fg
\qquad\text{given by}\qquad
(x,B) \mapsto x,
\]
the so-called Grothendieck--Springer simultaneous resolution. 

Let $\cN\subset\fg$ be the nilpotent cone.  Let also $\tcN=\pi^{-1}(\cN)$, and let $\mu:\tcN\to\cN$ be the restriction of $\pi$ to $\tcN$, the so-called Springer resolution.  
Let $i_{\cN}:\cN\to\fg$ and $i_{\tcN}:\tcN\to\tfg$ denote the closed embeddings.  Let $\fg_\rs\subset\fg$ be the open subvariety consisting of regular semisimple elements, $j_\rs : \fg_\rs \hookrightarrow \fg$ the inclusion, $\tfg_\rs := \pi^{-1}(\fg_\rs)$, and $\pi_\rs : \tfg_\rs \to \fg_\rs$ the restriction of $\pi$.

Fix, once and for all, a $G$-invariant nondegenerate bilinear form on $\fg$, and set
\[
N := (\dim \cN)/2, \qquad r := \dim \fH, \qquad d := \dim \fg = 2N+r.
\]
Then the rank of the vector bundle $\tfg\to\cB$ is $\dim B=N+r=d-N$.

\subsection{Rings and sheaves: conventions and notation}

All rings 
in the paper will be tacitly assumed to be noetherian commutative rings of finite global dimension.  If $\bk$ is such a ring, then a `$\bk$-algebra' should 
be understood to be noetherian, commutative, and of finite global dimension as well, but not necessarily finitely generated over $\bk$. Similarly, all our topological spaces will be tacitly assumed to be locally compact and of finite $c$-soft dimension in the sense of \cite[Exercise~II.9]{kas}. 

Let $\bk\allmod$ be the 
category of $\bk$-modules.  If $X$ is a topological space,
we write $\Sh(X,\bk)$ for the 
category of sheaves of $\bk$-modules on $X$, and $\Db(X,\bk)$ for its bounded derived category.  
(The objects in these categories are not required to be constructible.)  Given $M \in \bk\allmod$, let $\ub M_X$ (or simply $\ub M$) 
be the associated constant sheaf on $X$. 
We regard this operation as an exact functor
\begin{equation}
\label{eqn:constant-sheaf}
\ubb_X: \bk\allmod \to \Sh(X,\bk).
\end{equation}

If $f : X \to Y$ is a continuous map, we have (derived) functors
\[
f_*, f_! : \Db(X,\bk) \to \Db(Y,\bk) \qquad \text{and} \qquad f^*, f^! : \Db(Y,\bk) \to \Db(X,\bk).
\]
We will denote the nonderived versions of $f_!$ and $f^*$ by $\uo f_!$ and $\uo f^*$, respectively.

If $f : X \to Y$ is a continuous map, there exists an obvious isomorphism of functors which will be denoted by $\Const$:
\[
f^* \circ \ubb_Y \cong \ubb_X,
\]

If we have continuous maps of topological spaces
$X \xrightarrow{f} Y \xrightarrow{g} Z$
and if we let $h = g \circ f$, there are natural isomorphisms which will be denoted $\Co$:
\[
g_! \circ f_! \simto h_!
\qquad\text{and}\qquad
f^* \circ g^* \simto h^*.
\]

Finally, if we have a cartesian square:
\[
\xymatrix@R=0.5cm{
W \ar[r]^{g'} \ar[d]_{f'} \ar@{}[dr]|{\square} & X \ar[d]^{f} \\
Y \ar[r]_g & Z}
\]
we will denote by $\BC$ the base-change isomorphism
\[
g^* \circ f_! \simto f'_! \circ g'{}^*.
\]

\subsection{Change of rings}
\label{subsect:changerings}

If $\bk'$ is a $\bk$-algebra, we have exact forgetful functors
\[
\Formod = \Formod_{\bk,\bk'}: \bk'\allmod \to \bk\allmod
\qquad\text{and}\qquad
\For^X_{\bk,\bk'}: \Sh(X,\bk') \to \Sh(X,\bk).
\]
The derived functor of the latter will be denoted by the same symbol.

If $f: X \to Y$ is a continuous map of topological spaces, we claim that there is a natural isomorphism
\begin{equation}\label{eqn:for_!}
\For^Y_{\bk,\bk'} \circ f_! \simto f_! \circ \For^X_{\bk,\bk'}.
\end{equation}
To construct this isomorphism, we first note that the nonderived analogue $\For^Y_{\bk,\bk'} \circ \uo f_! \simto \uo f_! \circ \For^X_{\bk,\bk'}$ is obvious. The left hand side of \eqref{eqn:for_!} is canonically isomorphic to the right derived functor of $\For^Y_{\bk,\bk'} \circ \uo f_!$. Then, since $\For^X_{\bk,\bk'}$ takes $c$-soft $\bk'$-sheaves to $c$-soft $\bk$-sheaves (which are acyclic for $f_!$), the right hand side of \eqref{eqn:for_!} is canonically isomorphic to the right derived functor of $\uo f_! \circ \For^X_{\bk,\bk'}$ by \cite[Proposition~1.8.7]{kas}. The claim follows.
By a similar 
argument, we also obtain a canonical isomorphism
\begin{equation}\label{eqn:for^*}
\For^Y_{\bk,\bk'} \circ f^* \simto f^* \circ \For^X_{\bk,\bk'}.
\end{equation}
Finally, we have an obvious isomorphism
\begin{equation}\label{eqn:for-constant}
\For^X_{\bk,\bk'} \circ \ubb_X \simto \ubb_X \circ \Formod_{\bk,\bk'}.
\end{equation}

Appendix \ref{sect:isomsheaf} collects 
lemmas expressing the compatibility of these isomorphisms with each other and with various other 
isomorphisms that will be introduced later.

\subsection{`Perverse' sheaves}
\label{subsect:perverse}

Let $X$ be a complex algebraic variety, equipped with a 
stratification by locally closed smooth subvarieties.  
By~\cite[Corollaire~2.1.4]{bbd}, there is a well-defined `perverse $t$-structure' on $\Db(X,\bk)$ (for the middle perversity).  Its heart will be denoted $\Perv(X,\bk)$.  
Not all objects of $\Perv(X,\bk)$ are perverse sheaves in the conventional sense, because they are not required to be constructible.  Note that this category depends on the choice of stratification.

In particular, we equip $\cN$ with the stratification by $G$-orbits, and we equip $\fg$ with the Lie algebra analogue of the stratification given in~\cite[\S 3]{lusicc}.
In this stratification, the regular semisimple set forms a single stratum.  Each $G$-orbit in $\cN$ is also a stratum of $\fg$, so the functor $i_{\cN*}: \Db(\cN,\bk) \to \Db(\fg,\bk)$ restricts to an exact functor $i_{\cN*}: \Perv(\cN,\bk) \to \Perv(\fg,\bk)$.

\subsection{Springer and Grothendieck functors}

We define two additive functors: 
\begin{align*}
\Spr^{\bk} &: \bk\allmod \to \Perv(\cN,\bk) &&\text{by} & \Spr(M) &= \mu_!(\ub M_{\tcN})[2N], \\
\Groth^{\bk} &: \bk\allmod \to \Perv(\fg,\bk) &&\text{by} & \Groth(M) &= \pi_!(\ub M_{\tfg})[d].
\end{align*}
(Sometimes, for simplicity we write $\Spr$, $\Groth$ instead of $\Spr^{\bk}$, $\Groth^{\bk}$.)
The fact that these functors take values in the heart of the perverse $t$-structure follows 
from the fact that $\mu$ is semismall and $\pi$ is small.  
These functors
are exact functors of abelian categories. 
For any $M$ in $\bk\allmod$ we have a canonical isomorphism
\begin{equation}
\label{eqn:Groth-IC}
\Groth(M) \cong (j_\rs)_{!*} \bigl( (\pi_{\rs})_! \ub M_{\tfg_{\rs}} [d] \bigr).
\end{equation}

\begin{lem}\label{lem:groth-res}
The functor
$i_\cN^* \circ \Groth[-r] : \bk\allmod \to \Db(\cN,\bk)$
takes values in $\Perv(\cN,\bk)$. Moreover, there exists a natural isomorphism 
\begin{equation}\label{eqn:groth-res}
i_\cN^* \circ \Groth[-r] \simto \Spr.
\end{equation}
\end{lem}

\begin{proof}
Consider the cartesian diagram
\[
\vcenter{\xymatrix@R=0.6cm{
\tcN \ar[r]^-{i_{\tcN}} \ar[d]_-{\mu} \ar@{}[dr]|{\square} & \tfg \ar[d]^-{\pi} \\
\cN \ar[r]^{i_\cN} & \fg.
}}
\]
Then for any $M$ in $\bk\allmod$ we have functorial isomorphisms
\[
i_\cN^* \Groth(M)[-r] = i_\cN^* \pi_!(\ub M_{\tfg})[2N] \overset{\BC}{\cong} \mu_! i_{\tcN}^*(\ub M_{\tfg})[2N] \overset{\Const}{\cong} \mu_! (\ub M_{\tcN})[2N],
\]
which proves the claim.
\end{proof}

\subsection{Springer and Grothendieck functors and change of rings}
\label{ss:Spr-Groth-change-rings}

If $\bk'$ is a $\bk$-algebra, then we also define
\begin{align*}
\Spr^{\bk'}_{\bk} = \Spr^{\bk} \circ \Formod_{\bk,\bk'} &: \bk'\allmod \to \Perv(\cN,\bk), \\
\Groth^{\bk'}_{\bk} = \Groth^{\bk} \circ \Formod_{\bk,\bk'} &: \bk'\allmod \to \Perv(\fg,\bk).
\end{align*}
Using isomorphisms \eqref{eqn:for_!} and \eqref{eqn:for-constant} we obtain canonical isomorphisms of functors
\begin{equation}
\label{eqn:Spr-k-k'}
\Spr^{\bk'}_{\bk} \cong \For^{\cN}_{\bk,\bk'} \circ \Spr^{\bk'}, \qquad
\Groth^{\bk'}_{\bk} \cong \For^{\fg}_{\bk,\bk'} \circ \Groth^{\bk'}.
\end{equation}
Hence we can define two isomorphisms of functors
\begin{equation}
\label{eqn:groth-res-k-k'}
\phi : i_\cN^* \circ \Groth^{\bk'}_{\bk}[-r] \simto \Spr^{\bk'}_{\bk}, \quad \phi' : i_\cN^* \circ \Groth^{\bk'}_{\bk}[-r] \simto \Spr^{\bk'}_{\bk}
\end{equation}
in such a way that the following diagrams commute:
\[
\xymatrix@R=0.6cm{
 i^*_{\cN}  \Groth^{\bk} \Formod_{\bk,\bk'} [-r] \ar@{=}[d] \ar[r]^-{\eqref{eqn:groth-res}^{\bk}} & \Spr^{\bk} \Formod_{\bk,\bk'} \ar@{=}[d] \\
i^*_{\cN} \Groth^{\bk'}_{\bk}[-r] \ar[r]^-{\phi}
  & \Spr^{\bk'}_{\bk} } 
  \quad
\xymatrix@R=0.6cm{
\For^{\cN}_{\bk,\bk'} i^*_{\cN} \Groth^{\bk'}[-r] \ar[r]^-{\eqref{eqn:groth-res}^{\bk'}} \ar[d]_-{\eqref{eqn:for^*} \&
\eqref{eqn:Spr-k-k'}
}
  & \For^{\cN}_{\bk,\bk'} \Spr^{\bk'} \ar[d]^-{\eqref{eqn:Spr-k-k'}} \\
i^*_{\cN} \Groth^{\bk'}_{\bk}[-r] \ar[r]^-{\phi'}
  & \Spr^{\bk'}_{\bk} }
\]

Using Lemma \ref{lem:bc-for} and Lemma \ref{lem:const-for}, one can easily check that $\phi=\phi'$.
Hence from now on we will use the reference \eqref{eqn:groth-res-k-k'} to denote either $\phi$ or $\phi'$.

\subsection{Fourier--Sato transform}
\label{ss:fourier-sato-defn}

Let $Y$ be a topological space, and $p: E \to Y$ a complex vector bundle.  Consider the $\C^\times$-action on the fibers of $p$ given by scaling.  Recall that a sheaf on $E$ (resp.~an object of $\Db(E,\bk)$) is said to be \emph{conic} if its restrictions to $\C^\times$-orbits are locally constant (resp.~its cohomology sheaves are co\-nic).  Let $\Dcon(E,\bk) \subset \Db(E,\bk)$ be the full 
subcategory consisting of conic objects.

Let $\check p: E^* \to Y$ be the dual vector bundle. (Here $E^*$ is the dual vector bundle 
as a \emph{complex} vector bundle. However it can also be regarded as the dual to the \emph{real} vector bundle $E$ via the pairing $E \times E^* \to \mathbb{R}$ sending $(x,y)$ to $\mathrm{Re}(\langle x,y \rangle)$.)  Let
\[
Q = \{ (x,y) \in E \times_Y E^* \mid \mathrm{Re}(\langle x, y\rangle) \le 0 \} \subset E \times_Y E^*.
\]
Let $q: Q \to E$ and $\check q: Q \to E^*$ be the obvious projection maps.  The \emph{Fourier--Sato transform} is defined to be the functor
\begin{equation}\label{eqn:fourier-defn}
\bT_E : \Dcon(E,\bk) \simto \Dcon(E^*,\bk)
\qquad\text{given by}\qquad
\bT_E = \check q_! q^*[\rk E].
\end{equation}
This functor is an equivalence of categories;  see \cite{kas, bry}.  Note that the shift by $\rk E$ does not appear in the definition in~\cite{kas}.  Thus, with the notation of \emph{loc.}~\emph{cit.}, we have $\bT_E(M) = M^{\wedge}[\rk E]$.

If $\bk'$ is a $\bk$-algebra, from the definition~\eqref{eqn:fourier-defn}, we obtain a natural isomorphism
\begin{equation}\label{eqn:for-fourier}
\For^{E^*}_{\bk,\bk'} \circ \bT_E \simto \bT_E \circ \For^E_{\bk,\bk'}
\end{equation}
by combining~\eqref{eqn:for_!} and~\eqref{eqn:for^*}.

We will use several compatibility properties of the Fourier--Sato transform, which we recall now. (Lemmas asserting the compatibility between the various isomorphisms we introduce are to be found in Appendix~\ref{sect:isomsheaf}.)
Let $y: Y \to E^*$ be the embedding of the zero section. 
In \S\ref{ss:fourier-constant} we will construct a natural isomorphism
\begin{equation}\label{eqn:fourier-con}
\bT_E\circ \ubb_E \simto y_!\circ \ubb_Y[-\rk E].
\end{equation}
One of the steps of this construction is an isomorphism proved in Lemma \ref{lem:cohc}:
\begin{equation}
\label{eqn:cohc}
p_! \ubb_E \cong \ubb_Y[-2\rk E].
\end{equation}

Let $g: Y' \to Y$ be a continuous map of topological spaces, and form the cartesian squares
\[
\xymatrix@R=0.6cm{
E' = Y' \times_Y E \ar[r]^-f' \ar[d] \ar@{}[dr]|{\square} & E \ar[d] \\
Y' \ar[r]^g & Y}
\qquad
\xymatrix@R=0.6cm{
(E')^* = Y' \times_Y E^* \ar[r]^-{f'} \ar[d] \ar@{}[dr]|{\square} & E^* \ar[d] \\
Y' \ar[r]^g & Y}
\]
According to~\cite[Proposition~3.7.13]{kas}, there is a natural isomorphism
\begin{equation}\label{eqn:fourier_!}
\bT_E \circ f_! \simto f'_! \circ \bT_{E'}.
\end{equation}

Next, let $E_1 \to Y$ and $E_2 \to Y$ be two complex vector bundles, and let $\phi: E_1 \to E_2$ be a map of vector bundles.  Let $\ut \phi: E_2^* \to E_1^*$ denote the dual map.  According to~\cite[Proposition~3.7.14]{kas}, there is a natural isomorphism
\begin{equation}\label{eqn:fourier^*}
(\ut \phi)^* \circ \bT_{E_1} [-\rk E_1] \simto \bT_{E_2} \circ \phi_! [-\rk E_2].
\end{equation}

A particularly important special case arises when $E_2=Y$ is the trivial vector bundle over $Y$, so that $\phi$ is the bundle projection $p:E_1=E\to Y$. Then the dual map $\ut\phi$ is the embedding of the zero section $y:Y\to E^*$, so \eqref{eqn:fourier^*} becomes
\begin{equation}\label{eqn:zero-section}
y^* \circ \bT_{E} \simto p_! [\rk E]. 
\end{equation}
Here we have identified the Fourier--Sato transform for the trivial vector bundle with the identity functor. 

We will use in particular $\bT_E$ in the case $E=\fg$ (considered as a vector bundle over a point). Since $\fg$ is equipped with a nondegenerate bilinear form, we can identify it with its dual, and regard $\bT_\fg$ as an autoequivalence
$\bT_\fg : \Dcon(\fg,\bk) \simto \Dcon(\fg,\bk)$.

\subsection{Fourier--Sato transform and Springer and Grothendieck functors}
\label{ss:fourier-sato}

\begin{lem}
\label{lem:groth-fourier}
There is a natural isomorphism of functors
\begin{equation}\label{eqn:groth-fourier}
\bT_\fg \circ \Groth \cong i_{\cN!} \circ \Spr.
\end{equation}
\end{lem}

\begin{proof}
Consider $\fg\times\cB$ as a vector bundle over $\cB$. The functor $\bT_{\fg \times \cB}$ can be regarded as an autoequivalence of $\Dcon(\fg \times \cB,\bk)$.  Let $x : \tfg \hookrightarrow \fg \times \cB$ be the inclusion, which is a morphism of vector bundles over $\cB$, and let $\ut x : \fg \times \cB \to (\tfg)^*$ be the dual morphism. Let $y : \cB \hookrightarrow (\tfg)^*$ be the inclusion of the zero section. Finally, let $f:\fg \times \cB \to \fg$ be the projection. 
We observe that the following diagram is cartesian, where $z$ is the natural projection:
\[
\xymatrix@R=0.6cm{
\tcN \ar[r]^-{z} \ar[d]_-{x \circ i_{\tcN}} \ar@{}[dr]|{\square} & \cB\ar[d]^-{y} \\
\fg \times \cB \ar[r]^-{\ut x} & (\tfg)^*.
}
\]

Now we are in a position to prove the lemma. For $M$ in $\bk\allmod$ we have
\[
\bT_\fg \Groth(M) = \bT_\fg \pi_! (\ub M_{\tfg})[d] \overset{\Co}{\cong} \bT_\fg f_!x_!(\ub M_{\tfg})[d].
\]
We deduce isomorphisms
\begin{multline*}
\bT_\fg \Groth(M) \overset{\eqref{eqn:fourier_!}}{\cong} f_! \bT_{\fg \times \cB} x_!(\ub M_{\tfg})[d] \overset{\eqref{eqn:fourier^*}}{\cong} f_! (\ut x)^* \bT_{\tfg}(\ub M_{\tfg})[d+N] \\
\overset{\eqref{eqn:fourier-con}}{\cong}  f_! (\ut x)^* y_!(\ub M_\cB)[2N] \overset{\BC}{\cong} f_!(x \circ i_{\tcN})_! z^*(\ub M_\cB)[2N].
\end{multline*}
Using the $\Const$ isomorphism $z^*(\ub M_\cB) \cong \ub M_{\tcN}$ and the fact that $f \circ x \circ i_{\tcN} = i_{\cN}\circ\mu$, we deduce the isomorphism of the lemma.
\end{proof}

\subsection{Fourier--Sato transform, Springer and Grothendieck functors, and change of rings}
\label{ss:fourier-spr-groth-change-rings}

Let $\bk'$ be a $\bk$-algebra.
Recall the functors $\Groth^{\bk'}_{\bk}$ and $\Spr^{\bk'}_{\bk}$ of \S\ref{ss:Spr-Groth-change-rings}. As in \emph{loc.}~\emph{cit.}~there exist two natural isomorphisms of functors
\begin{equation}
\label{eqn:morphism-Fourier-k-k'}
\psi : \bT_\fg \circ \Groth^{\bk'}_{\bk} \simto i_{\cN!} \circ \Spr^{\bk'}_{\bk}, \quad \psi' : \bT_\fg \circ \Groth^{\bk'}_{\bk} \simto i_{\cN!} \circ \Spr^{\bk'}_{\bk}
\end{equation}
defined in such a way that the following diagrams commute:
\[
\xymatrix@R=0.6cm{
\bT_\fg  \Groth^{\bk} \Formod_{\bk,\bk'} \ar@{=}[d] \ar[r]^-{\eqref{eqn:groth-fourier}_\bk} & i_{\cN!} \Spr^{\bk} \Formod_{\bk,\bk'} \ar@{=}[d] \\
\bT_\fg \Groth^{\bk'}_{\bk} \ar[r]^-{\psi}
  & i_{\cN!} \Spr^{\bk'}_{\bk} } 
  \
\xymatrix@C=1cm@R=0.6cm{
\For^{\fg}_{\bk,\bk'} \bT_\fg \Groth^{\bk'} \ar[r]^-{\eqref{eqn:groth-fourier}_{\bk'}} \ar[d]_-{\eqref{eqn:for-fourier} \& 
\eqref{eqn:Spr-k-k'}}
  & \For^{\fg}_{\bk,\bk'} i_{\cN!} \Spr^{\bk'} \ar[d]_-{\eqref{eqn:for_!} \& \eqref{eqn:Spr-k-k'}} \\
\bT_\fg \Groth^{\bk'}_{\bk} \ar[r]^-{\psi'}
  & i_{\cN!} \Spr^{\bk'}_{\bk} }
\]

Using Lemmas~\ref{lem:bc-for}, ~\ref{lem:fourier-for_!}, \ref{lem:fourier-for^*}, and \ref{lem:fourier-constant}, one can easily check that $\psi=\psi'$.
Hence from now on we will use the reference \eqref{eqn:morphism-Fourier-k-k'} to denote either $\psi$ or $\psi'$.

\section{Endomorphisms under change of rings}
\label{sect:restrict}

\subsection{Functors commuting with direct sums}

Let $\sC$ be an abelian category which admits arbitrary direct sums. Recall that a functor $F : \bk\allmod \to \sC$ \emph{commutes with direct sums} if for any set $I$ and any collection $(M_i)_{i \in I}$ of $\bk$-modules the following natural morphism is an isomorphism:
\begin{equation}
\label{eqn:limit-F}
\bigoplus_{i \in I} F(M_i) \to F\bigl( \bigoplus_{i \in I} M_i \bigr).
\end{equation}

\begin{lem}\label{lem:end-functor}
Let $F : \bk\allmod \to \sC$ be an exact 
functor which commutes with direct sums, and let $\End(F)$ be the 
algebra of endomorphisms of $F$ in the category of functors from $\bk\allmod$ to $\sC$. Then the algebra morphism
\[
\mathrm{ev}_\bk :
\End(F) \to \End_{\sC}(F(\bk))
\]
given by evaluation at $\bk$ is an isomorphism.
\end{lem}

\begin{proof}
First we note that if $\eta$ is in $\End(F)$, then for any set $I$ and any collection $(M_i)_{i \in I}$ of $\bk$-modules, the following diagram commutes:
\[
\xymatrix@C=1.5cm@R=0.6cm{
\bigoplus_{i \in I} F(M_i) \ar[d]_-{\eqref{eqn:limit-F}}^-{\wr} \ar[r]^-{\oplus_i \eta_{M_i}} & \bigoplus_{i \in I} F(M_i) \ar[d]^-{\eqref{eqn:limit-F}}_-{\wr} \\
F\bigl( \bigoplus_{i \in I} M_i \bigr) \ar[r]^-{\eta_{\oplus_i M_i}} & F\bigl( \bigoplus_{i \in I} M_i \bigr).
}
\]

Let us prove injectivity of our morphism. Assume that $\eta_\bk=0$. Let $M$ be in $\bk\allmod$, and choose a set $I$ and a surjection $\bk^{\oplus I} \twoheadrightarrow M$. Then as $F$ is exact we have a surjection $F(\bk^{\oplus I}) \twoheadrightarrow F(M)$. Using the diagram above one checks that $\eta_{\bk^{\oplus I}}=0$, which implies that $\eta_M=0$.

Now we prove surjectivity. Let $\zeta : F(\bk) \to F(\bk)$ be a morphism in $\sC$. If $M$ is in $\bk\allmod$ we choose a presentation $\bk^{\oplus J} \xrightarrow{f_1} \bk^{\oplus I} \xrightarrow{f_2} M \to 0$
for some sets $I$ and $J$. The image under $F$ of this presentation is an exact sequence. 
Then we define $\eta_M$ as the unique morphism which makes the following diagram commutative, where we use the canonical isomorphisms $F(\bk^{\oplus I}) \cong F(\bk)^{\oplus I}$ and $F(\bk^{\oplus J}) \cong F(\bk)^{\oplus J}$: 
\[
\xymatrix@C=1.5cm@R=0.6cm{
F(\bk)^{\oplus J} \ar[r]^-{F(f_1)} \ar[d]_-{\zeta^{\oplus J}} & F(\bk)^{\oplus I} \ar[d]_-{\zeta^{\oplus I}} \ar@{->>}[r]^-{F(f_2)} & F(M) \ar[d]^-{\eta_M} \\
F(\bk)^{\oplus J} \ar[r]^-{F(f_1)} & F(\bk)^{\oplus I} \ar@{->>}[r]^-{F(f_2)} & F(M).
}
\]
To check that $\eta_M$ does not depend on the presentation, and that it defines a morphism of functors, is an easy exercise left to the reader.
\end{proof}

\subsection{Endomorphisms}

As above we denote by $\End(\Spr)$ the $\bk$-algebra of endomorphisms of the functor $\Spr$,
 and likewise for $\End(\Groth)$. 

\begin{lem}\label{lem:end-gen}
The following $\bk$-algebra homomorphisms are isomorphisms:
\begin{gather}
\label{eqn:end-spr}
\End(\Spr) \xrightarrow{\mathrm{ev}_\bk} \End_{\Perv(\cN,\bk)}(\Spr(\bk)), \\
\label{eqn:end-groth}
\End(\Groth) \xrightarrow{\mathrm{ev}_\bk} \End_{\Perv(\fg,\bk)}(\Groth(\bk)) \xrightarrow{j_\rs^*} \End_{\Perv(\fg_\rs,\bk)}(\Groth(\bk)|_{\fg_\rs}).
\end{gather}
\end{lem}

\begin{proof}
The functors $\Groth$ and $\Spr$ both commute with direct sums: this follows from the observation that the functor \eqref{eqn:constant-sheaf} commutes with direct sums, as well as the functors $\pi_!$ and $\mu_!$. Using Lemma \ref{lem:end-functor}, we deduce that \eqref{eqn:end-spr} and the first morphism of \eqref{eqn:end-groth} are isomorphisms. The fact that the second morphism of \eqref{eqn:end-groth} is also an isomorphism follows from \eqref{eqn:Groth-IC} and the well-known property that the functor $(j_\rs)_{!*}$ is fully faithful (see e.g.~\cite[Proposition 2.29]{juteau-aif}).
\end{proof}

\subsection{Change of rings}
\label{subsect:changerings2}

The next two propositions deal with a comparison between two coefficient rings.  Let $\bk'$ be a $\bk$-algebra. Recall the functors $\Spr^{\bk'}_{\bk}$ and $\Groth^{\bk'}_{\bk}$ introduced in \S \ref{ss:Spr-Groth-change-rings}. Composing on the right with $\Formod_{\bk,\bk'}$, respectively composing on the left with $\For^\fg_{\bk,\bk'}$ and using isomorphism
\eqref{eqn:Spr-k-k'}, we obtain natural maps
\[
\End(\Groth^\bk) \to \End(\Groth^{\bk'}_\bk) \leftarrow \End(\Groth^{\bk'}),
\]
and likewise for $\Spr$. The map $\End(\Groth^{\bk'}_\bk) \to \End(i_\cN^* \circ \Groth^{\bk'}_\bk)$ induced by $i_\cN^*$, combined with isomorphism~\eqref{eqn:groth-res-k-k'}, gives a morphism of $\bk$-algebras
\[
\tilde\rho^{\bk'}_\bk: \End(\Groth^{\bk'}_\bk) \to \End(\Spr^{\bk'}_\bk).
\]
Similar constructions using isomorphism \eqref{eqn:groth-res} provide $\bk$-algebra morphisms
\[
\tilde\rho^\bk : \End(\Groth^\bk) \to \End(\Spr^\bk) \qquad\text{and}\qquad \tilde\rho^{\bk'} : \End(\Groth^{\bk'}) \to \End(\Spr^{\bk'}).
\]

Similarly, the map $\End(\Groth^{\bk'}_\bk) \to \End(\bT_{\fg} \circ \Groth^{\bk'}_\bk)$ induced by $\bT_\fg$, combined with isomorphisms \eqref{eqn:groth-fourier} and \eqref{eqn:morphism-Fourier-k-k'} and the fact that $i_{\cN!}$ is fully faithful, gives rise to morphisms of $\bk$-algebras
\begin{gather*}
\tilde\varphi^{\bk'}_\bk: \End(\Groth^{\bk'}_\bk) \to \End(\Spr^{\bk'}_\bk),
\\
\tilde\varphi^\bk: \End(\Groth^\bk) \to \End(\Spr^\bk), \qquad \tilde\varphi^{\bk'}: \End(\Groth^{\bk'}) \to \End(\Spr^{\bk'}).
\end{gather*}
Since $\bT_\fg$ is an equivalence of categories, these morphisms are algebra isomorphisms.

\begin{lem}\label{lem:scalar-compare}
The following diagrams commute, where the vertical arrows are defined above:
\[
\xymatrix@R=0.6cm{
\End(\Groth^{\bk}) \ar[r]^{\tilde\rho^\bk} \ar[d] & \End(\Spr^{\bk}) \ar[d] \\
\End(\Groth^{\bk'}_{\bk}) \ar[r]^{\tilde\rho^{\bk'}_\bk} & \End(\Spr^{\bk'}_{\bk}) \\
\End(\Groth^{\bk'}) \ar[r]^{\tilde\rho^{\bk'}} \ar[u] & \End(\Spr^{\bk'}) \ar[u]}
\qquad\qquad
\xymatrix@R=0.6cm{
\End(\Groth^{\bk}) \ar[r]^{\tilde\varphi^\bk} \ar[d] & \End(\Spr^{\bk}) \ar[d] \\
\End(\Groth^{\bk'}_{\bk}) \ar[r]^{\tilde\varphi^{\bk'}_\bk} & \End(\Spr^{\bk'}_{\bk}) \\
\End(\Groth^{\bk'}) \ar[r]^{\tilde\varphi^{\bk'}} \ar[u] & \End(\Spr^{\bk'}) \ar[u]}
\]
\end{lem}
\begin{proof}
Consider the left-hand diagram. The commutativity of the upper square easily follows from the definition of \eqref{eqn:groth-res-k-k'} as the morphism $\phi$ of \S\ref{ss:Spr-Groth-change-rings}, and the commutativity of the lower square easily follows from the equivalent definition of this morphism as the morphism $\phi'$ of \S\ref{ss:Spr-Groth-change-rings}. The proof of the commutativity of the right-hand diagram is similar, using 
the morphisms $\psi$ and $\psi'$ of 
\S\ref{ss:fourier-spr-groth-change-rings}.
\end{proof}

\begin{lem}\label{lem:scalar-injective}
The maps 
\[
\End(\Groth^{\bk'}) \to \End(\Groth^{\bk'}_\bk) \qquad\text{and}\qquad \End(\Spr^{\bk'}) \to \End(\Spr^{\bk'}_\bk)
\]
defined above are injective.
\end{lem}
\begin{proof}
For the first assertion, by Lemma~\ref{lem:end-gen}, it suffices to show that the map
\[
\End_{\Perv(\fg_\rs,\bk')}(\Groth^{\bk'}(\bk')|_{\fg_\rs}) \to
\End_{\Perv(\fg_\rs,\bk)}(\Groth^{\bk'}_{\bk}(\bk')|_{\fg_\rs})
\]
is injective.  But this is clear, since we are now comparing two endomorphism rings of a single locally constant sheaf on $\fg_\rs$.  The result for $\Spr$ then follows from the second diagram in Lemma~\ref{lem:scalar-compare}, using the fact that the morphisms $\tilde\varphi$ are isomorphisms.
\end{proof}

\subsection{$\bW$-action on $\Groth$}

Let us now construct a canonical isomorphism
\[
r_{\bk}: \bk[\bW] \simto \End(\Groth^{\bk}).
\]
Throughout this construction, the notation $\ub M$ will indicate a constant sheaf on $\tfg_\rs$.  Because $\pi_\rs$ is a regular covering map whose group of deck transformations is $\bW$, there is a natural action of $\bW$ on any object of the form $\pi_{\rs *} \ub M$.  Since all our objects live in $\bk$-linear categories, this extends to a natural action of the ring $\bk[\bW]$. Moreover, the morphism $\tilde r_\bk : \bk[\bW] \to \End_{\Sh(\fg_\rs,\bk)}(\pi_{\rs *} \ub \bk)$ induced by this action is an isomorphism. Combining this with Lemma \ref{lem:end-gen} provides the isomorphism $r_\bk$.

\begin{prop}\label{prop:groth-end}
If $\bk'$ is a $\bk$-algebra, the morphism $\End(\Groth^\bk) \to \End(\Groth^{\bk'}_\bk)$ defined in {\rm \S\ref{subsect:changerings2}} factors (uniquely) through a morphism
\[
\alpha^{\bk'}_\bk : \End(\Groth^\bk) \to \End(\Groth^{\bk'}).
\]
Moreover the following diagram commutes, where the left vertical morphism is the natural algebra morphism:
\[
\xymatrix@R=0.6cm{
\bk[\bW] \ar[d]\ar[r]^-{r_\bk} & \End(\Groth^\bk) \ar[d]^-{\alpha^{\bk'}_\bk} \\
\bk'[\bW] \ar[r]_-{r_{\bk'}} & \End(\Groth^{\bk'})}
\]
\end{prop}

\begin{proof}
Consider the following diagram:
\[
\xymatrix@C=0.6cm@R=0.6cm{
\bk[\bW] \ar[r]_-{\sim}^-{\tilde{r}_\bk} \ar[dd] & \End_{\Sh(\fg_\rs,\bk)}(\pi_{\rs *} \ub \bk) & \End_{\Perv(\fg,\bk)}(\Groth^{\bk}(\bk)) \ar[l]^-{\sim}_-{c} & \End(\Groth^{\bk}) \ar[l]^-{\sim}_-{\mathrm{ev}_\bk} \ar[d]^-{a} \ar[ld]_-{\mathrm{\mathrm{ev}_{\bk'}}} \\
& \End_{\Sh(\fg_\rs,\bk)}(\pi_{\rs *} \ub \bk') & \End_{\Perv(\fg,\bk)}(\Groth^{\bk}(\bk')) \ar[l]^-{\sim}_-{c} & \End(\Groth^{\bk'}_{\bk}) \ar[l]^-{\sim}_-{\mathrm{ev}_{\bk'}} \\
\bk'[\bW] \ar[r]_-{\sim}^-{\tilde r_{\bk'}} & \End_{\Sh(\fg_\rs,\bk')}(\pi_{\rs *} \ub \bk') \ar[u]_-{b} & \End_{\Perv(\fg,\bk')}(\Groth^{\bk'}(\bk'))  \ar[l]^-{\sim}_-{c} \ar[u]_-{b} & \End(\Groth^{\bk'}) \ar[l]^-{\sim}_-{\mathrm{ev}_{\bk'}} \ar[u]_-{b}}
\]
The map labelled ``$a$" is induced by the functor $\Formod_{\bk,\bk'}$, the maps 
``$b$'' are induced by 
$\For^\fg_{\bk,\bk'}$ or $\For^{\fg_\rs}_{\bk,\bk'}$, and the maps 
``$c$'' are induced by 
$j_\rs^*$. By Lemma \ref{lem:end-functor}, the horizontal maps from the fourth column to the third one are isomorphisms.

The two squares on the bottom of the diagram are obviously commutative, as is the triangle in the top right corner. Hence the diagram defines two maps from $\bk[\bW]$ to $\End_{\Sh(\fg_\rs,\bk)}(\pi_{\rs *} \ub \bk')$. We claim that these maps coincide: 
this follows from the fact that these maps are induced by the $\bW$-action on $\pi_{\rs *} \ub \bk'$, which is independent of the coefficient ring. 
Hence the diagram as a whole is commutative.

It follows that the image of $\End(\Groth^\bk) \xrightarrow{a} \End(\Groth^{\bk'}_\bk)$ is contained in the image of $\End(\Groth^{\bk'}) \xrightarrow{b} \End(\Groth^{\bk'}_\bk)$. Since the latter map is injective (see Lemma \ref{lem:scalar-injective}), the existence of $\alpha^{\bk'}_{\bk}$ follows. The commutativity of the diagram is clear. 
\end{proof}

\subsection{$\bW$-actions on $\Spr$}

We now define
\begin{align*}
\rho^\bk&: \bk[\bW] \to \End(\Spr^\bk) &&\text{by} & \rho^\bk &= \tilde\rho^\bk \circ r_\bk, \\
\varphi^\bk&: \bk[\bW] \to \End(\Spr^\bk) &&\text{by} & \varphi^\bk &= \tilde\varphi^\bk \circ r_\bk.
\end{align*}
Like $\tilde\varphi^\bk$, the map $\varphi^\bk$ is an isomorphism of $\bk$-algebras.  By composing with the isomorphism in Lemma~\ref{lem:end-gen}, we obtain the maps $\drho^\bk, \dvarphi^\bk$ considered in Theorem~\ref{thm:main-intro}.

\begin{prop}
\label{prop:scalars}
If $\bk'$ is a $\bk$-algebra, there is a canonical $\bk$-algebra homomorphism
\[
\beta^{\bk'}_\bk : \End(\Spr^\bk) \to \End(\Spr^{\bk'})
\]
such that the following diagrams commute:
\[
\xymatrix@R=0.5cm{
\bk[\bW] \ar[d]\ar[r]^-{\rho^\bk} & \End(\Spr^\bk) \ar[d]^-{\beta^{\bk'}_\bk} \\
\bk'[\bW] \ar[r]^-{\rho^{\bk'}} & \End(\Spr^{\bk'})}
\qquad\qquad
\xymatrix@R=0.5cm{
\bk[\bW] \ar[d]\ar[r]^-{\varphi^\bk} & \End(\Spr^\bk) \ar[d]^-{\beta^{\bk'}_\bk} \\
\bk'[\bW] \ar[r]^-{\varphi^{\bk'}} & \End(\Spr^{\bk'})}
\]
\end{prop}

\begin{proof}
Recall the morphisms considered in Lemma \ref{lem:scalar-compare}.
We claim first that the image of the morphism $\End(\Spr^\bk) \to \End(\Spr^{\bk'}_{\bk})$ is contained in the image of the map $\End(\Spr^{\bk'}) \to \End(\Spr^{\bk'}_{\bk})$.  Indeed, this follows from the corresponding assertion for $\Groth$ (proved in the course of establishing Proposition~\ref{prop:groth-end}) and the second commutative diagram in Lemma~\ref{lem:scalar-compare}, in which the horizontal maps are isomorphisms.  In view of Lemma~\ref{lem:scalar-injective}, we get a natural map $\beta^{\bk'}_\bk : \End(\Spr^\bk) \to \End(\Spr^{\bk'})$, and then the two commutative diagrams in Lemma~\ref{lem:scalar-compare} yield the simpler diagrams
\[
\vcenter{
\xymatrix{
\End(\Groth^{\bk}) \ar[r]^-{\tilde\rho^{\bk}} \ar[d]_-{\alpha^{\bk'}_\bk} & \End(\Spr^{\bk}) \ar[d]^-{\beta^{\bk'}_\bk} \\
\End(\Groth^{\bk'}) \ar[r]^-{\tilde\rho^{\bk'}} & \End(\Spr^{\bk'})}
}
\qquad\text{and}\qquad
\vcenter{
\xymatrix{
\End(\Groth^{\bk}) \ar[r]^-{\tilde\varphi^{\bk}} \ar[d]_-{\alpha^{\bk'}_\bk} & \End(\Spr^{\bk}) \ar[d]^-{\beta^{\bk'}_\bk} \\
\End(\Groth^{\bk'}) \ar[r]^-{\tilde\varphi^{\bk'}} & \End(\Spr^{\bk'})}
}
\]
We conclude by combining these diagrams with Proposition~\ref{prop:groth-end}.
\end{proof}

\subsection{Comparing the actions}

We can now explain how the main result of the paper, Theorem~\ref{thm:main-intro}, reduces to the $\bk=\Z$ case.  By Lemma~\ref{lem:end-gen}, Theorem~\ref{thm:main-intro} (for any given $\bk$) is equivalent to the following statement.

\begin{thm}\label{thm:main}
The two maps $\rho^\bk, \varphi^\bk: \bk[\bW] \to \End(\Spr^{\bk})$ are related as follows:
\[
\rho^\bk = \varphi^\bk \circ \sgn.
\]
\end{thm}

\noindent
Since we can regard $\bk$ as a $\Z$-algebra, it is obvious from Proposition~\ref{prop:scalars} that it suffices to prove Theorem~\ref{thm:main} in the $\bk=\Z$ case.
The $\bk=\Z$ case of Theorem~\ref{thm:main-intro} is covered by known arguments, explained in the next section.

\section{Proof of the `easy' case}
\label{sect:thm-Q}

Here we explain how to prove Theorem~\ref{thm:main-intro} in the `easy' case of \S\ref{ss:strategy}.
The required arguments 
appeared first around thirty years ago in papers by Hotta~\cite{hotta}, Brylinski~\cite{bry} and Spaltenstein~\cite{spalt}, and subsequently in the surveys by Shoji~\cite{shoji} and Jantzen~\cite{jantzen}. Those references considered the \'etale topology and $\ell$-adic cohomology, so some adaptations to our setting are needed. The other reason for revisiting these arguments is that our statement is slightly more precise than, for example, ~\cite[Proposition 17.7]{shoji}, in that it amounts to proving $\bW$-equivariance of a specific isomorphism rather than the mere existence of a $\bW$-equivariant isomorphism. 

\subsection{Actions on the cohomology of $\cB$}
\label{ss:cohomology-B}

We begin by recalling the classical $\bW$-action on $H^\bullet(\cB;\bk)$.  Choose a Borel subgroup $B \subset G$ and a maximal torus $T \subset B$; we obtain a canonical identification $\bW = N_G(T)/T$.  Then $\bW$ acts on $G/T$ and on $H^\bullet(G/T;\bk)$.  The map $G/T \to G/B = \cB$ is a fibre bundle whose fibre $B/T$ is an affine space of dimension $N$, and it thus induces an isomorphism
\begin{equation} \label{eqn:cohom}
H^\bullet(\cB;\bk) \cong H^\bullet(G/T;\bk).
\end{equation}
We use this isomorphism to define a $\bW$-action on $H^\bullet(\cB;\bk)$. Note that this construction is compatible with change of rings in the obvious sense: if $\bk'$ is a $\bk$-algebra then the natural morphism $H^\bullet(\cB;\bk) \to H^\bullet(\cB;\bk')$ is $\bW$-equivariant.
A similar construction defines a $\bW$-action on the Chow ring $A(\cB)$ of $\cB$ (see \cite[\S 4.7]{demazure-ens}).

\begin{lem}
\label{lem:Chow-cohomology}
There is a canonical isomorphism $\bk \otimes_\Z A(\cB) \cong H^\bullet(\cB;\bk)$ of graded $\bk[\bW]$-modules.
\end{lem}

\begin{proof}
Since $\cB$ is paved by affine spaces, $H^\bullet(\cB;\bk)$ is a free $\bk$-module of rank $|\bW|$, and the natural morphism $H^\bullet(\cB;\Z) \to H^\bullet(\cB;\bk)$ induces an isomorphism $\bk \otimes_\Z H^\bullet(\cB;\Z) \xrightarrow{\sim} H^\bullet(\cB;\bk)$. Hence it is enough to prove the lemma when $\bk=\Z$.

There is a canonical morphism $A(\cB) \to H^\bullet(\cB;\Z)$, see e.g.~\cite[\S 19.1]{fulton}. By functoriality this morphism is $\bW$-equivariant. The arguments in \cite[Proposition 13.1]{jantzen} prove that it is injective. Surjectivity follows from the fact that $H^\bullet(\cB;\Z)$ is spanned by characteristic classes of Schubert varieties. (See also \cite[\S 8]{demazure-inv}.)
\end{proof}

\begin{lem}
\label{lem:regular}
Let $\mathbb{F}$ be either $\Q$ or $\mathbb{F}_p$, where $p$ does not divide $|\bW|$. Then the $\mathbb{F}[\bW]$-module $\mathbb{F} \otimes_\Z A(\cB)$ is isomorphic to the regular representation.
\end{lem}

\begin{proof}
By \cite[\S 4.6]{demazure-ens}, $\Q \otimes_\Z A(\cB)$ is isomorphic to the coinvariant algebra $C_\Q$. Hence when $\mathbb{F}=\Q$ the result follows from~\cite[Ch.~V, \S5, Th\'eor\`eme 2]{bourbaki}. If $\mathbb{F}=\mathbb{F}_p$, recall that the composition factors of the modular reduction of a $\Q_p[\bW]$-module do not depend on the $\Z_p$-lattice used for the reduction, see~\cite[\S 15.2]{serre}. As $\Q_p \otimes_\Z A(\cB)$ is isomorphic to the regular representation, we deduce that $\mathbb{F} \otimes_\Z A(\cB)$ has the same composition factors as $\mathbb{F}[\bW]$. Under our assumptions $\mathbb{F}[\bW]$ is a semisimple ring, whence the claim. 
\end{proof}

We denote by $\bk_\mathrm{triv}$, respectively~$\bk_\varepsilon$, the $\bk[\bW]$-module which is free of rank one over $\bk$ with trivial $\bW$-action, respectively~with $\bW$-action defined by $\varepsilon$.

\begin{lem}
\label{lem:cohomology-B}
\begin{enumerate}
\item
The $\bk[\bW]$-module $H^{2N}(\cB;\bk)$ is isomorphic to $\bk_{\varepsilon}$.
\item
Assume that $\bk$ embeds either in a $\Q$-algebra, or in an $\F_p$-algebra where $p$ does not divide $|\bW|$.
Then the $\bk[\bW]$-module $H^\bullet(\cB;\bk)$ is faithful.
\end{enumerate}
\end{lem}

\begin{proof}
(1) By Lemma \ref{lem:Chow-cohomology}, it is enough to prove the analogous claim for $A(\cB)$. We already know that $A^{2N}(\cB)$ is a free $\Z$-module of rank $1$, and we only have to identify the $\bW$-action. However, the action of simple reflections is described in \cite[\S 4.7]{demazure-ens}, which is sufficient to prove the claim.

(2) 
It is enough to prove the claim when $\bk$ is a $\Q$-algebra, or an $\F_p$-algebra where $p$ does not divide $|\bW|$. However in this case, by Lemmas \ref{lem:Chow-cohomology} and \ref{lem:regular}, the $\bk[\bW]$-module $H^\bullet(\cB;\bk)$ is isomorphic to the regular representation, which is faithful.
\end{proof}

One can also consider cohomology with compact support. As above there is a canonical isomorphism
\begin{equation*} 
H^\bullet_c(\cB;\bk) \cong H^{\bullet+2N}_c(G/T;\bk),
\end{equation*}
which we can use to define a $\bW$-action on $H^\bullet_c(\cB;\bk)$. 

\begin{lem}
\label{lem:cohomology-c-B}
The $\bW$-module $H_c^{2N}(\cB;\bk)$ is isomorphic to $\bk_{\mathrm{triv}}$.
\end{lem}

\begin{proof}
By definition we have an isomorphism $H_c^{2N}(\cB;\bk) \cong H^{4N}_c(G/T;\bk)$ of $\bW$-modules. Now $H^{4N}_c(G/T;\bk)$ is free of rank one, with a canonical generator given by the orientation of $G/T$ induced by the complex structure. This generator is fixed by $\bW$, which proves the claim.
\end{proof}

Of course, since $\cB$ is compact, one would not normally distinguish $H^\bullet_c(\cB;\bk)$ from $H^\bullet(\cB;\bk)$. We do so here because the $\bW$-actions are not quite the same:

\begin{lem}\label{lem:poincare}
The natural isomorphism 
\begin{equation}
\label{eqn:poincare}
H^\bullet_c(\cB;\bk) \simto H^\bullet(\cB;\bk) 
\end{equation}
is $(\bW,\varepsilon)$-equivariant, in the sense that it becomes $\bW$-equivariant when one of the two sides is tensored with the sign representation of $\bW$.
\end{lem}

\begin{proof}
For any $i$ we have natural (in particular, $\bW$-equivariant) pairings giving the horizontal arrows of the commutative diagram
\[
\xymatrix@R=0.5cm{
H^i(\cB;\bk)\otimes_\bk H^{2N-i}_c(\cB;\bk) \ar[r] \ar[d] & H^{2N}_c(\cB;\bk) \ar[d] \\
H^i(\cB;\bk)\otimes_\bk H^{2N-i}(\cB;\bk) \ar[r] & H^{2N}(\cB;\bk)
}
\]
where the vertical arrows are given by \eqref{eqn:poincare}. By Poincar{\'e} duality, the upper arrow is a perfect pairing. (Note that $H^\bullet(\cB;\bk)$ is free, so that Poincar{\'e} duality holds with coefficients in $\bk$.) Hence the same is true for the lower arrow. By Lemma \ref{lem:cohomology-B} and Lemma \ref{lem:cohomology-c-B}, the right vertical map is $(\bW,\varepsilon)$-equivariant. The claim follows.
\end{proof}

\subsection{Comparison with Springer actions}

For any variety $X$, let $a_X$ denote the constant map $X \to \mathrm{pt}$.  Factoring $a_{\tfg}$ as either $a_\fg \circ \pi$ or as $a_\cB \circ t$ (where $t: \tfg \to \cB$ is the projection, a vector bundle of rank $\frac{1}{2}(d+r)$), we see that there are isomorphisms
\begin{equation}\label{eqn:groth-hcb}
H^\bullet \big( a_{\fg!}\Groth(\bk) \big) \simto H^\bullet \big( a_{\tfg!}\underline{\bk}_{\tfg}[d] \big) \simto H^{\bullet-r}_c(\cB;\bk).
\end{equation}
On the other hand, applying base-change to the cartesian square
\[
\vcenter{\xymatrix@R=0.6cm{
\cB \ar[r]^{j_\cB} \ar[d] \ar@{}[dr]|{\square} & \tfg \ar[d]^{\pi} \\
\{0\} \ar[r]_{j_0} & \fg}}
\]
(where all morphisms are the natural ones)
we obtain an isomorphism
\begin{equation}\label{eqn:groth-hob}
H^\bullet \big( j_0^*\Groth(\bk) \big) \simto H^{\bullet+d}(\cB;\bk).
\end{equation}

Let us denote by $\dr_\bk$ the composition of $r_\bk$ with the morphism $\mathrm{ev}_\bk : \End(\Groth) \to \End(\Groth(\bk))$. Both~\eqref{eqn:groth-hcb} and~\eqref{eqn:groth-hob} are $\bW$-equivariant, in the following sense.

\begin{lem}\label{lem:gamma-action}
\begin{enumerate}
\item Let $\gamma: \End(\Groth(\bk)) \to \End(H^\bullet(\cB;\bk))$ be the map induced by~\eqref{eqn:groth-hob}.  Then $\gamma \circ \dr_\bk: \bk[\bW] \to \End(H^\bullet(\cB;\bk))$ coincides with the $\bW$-action on $H^\bullet(\cB;\bk)$ defined in \S{\rm \ref{ss:cohomology-B}}.
\item Let $\gamma_c: \End(\Groth(\bk)) \to \End(H^\bullet_c(\cB;\bk))$ be the map induced by~\eqref{eqn:groth-hcb}.  Then $\gamma_c \circ \dr_\bk: \bk[\bW] \to \End(H^\bullet_c(\cB;\bk))$ coincides with the $\bW$-action on $H^\bullet_c(\cB;\bk)$ defined in \S{\rm \ref{ss:cohomology-B}}.
\end{enumerate}
\end{lem}
\begin{proof}
Spaltenstein's proof of the $\ell$-adic analogue of (1)~\cite[Theorem 2]{spalt}, which is recapitulated in \cite[Proposition 5.4]{shoji} and \cite[Proposition 13.7]{jantzen}, adapts directly to our setting by simply replacing $\ell$-adic cohomology with $\bk$-cohomology throughout. A large part of this adapted proof amounts to showing that the following variant of~\eqref{eqn:groth-hcb} is $\bW$-equivariant:
$H^\bullet \big( a_{\fg*}\Groth(\bk) \big) \simto H^\bullet \big( a_{\tfg*}\underline{\bk}_{\tfg}[d] \big) \simto H^{\bullet+d}(\cB;\bk)$.
The proof of (2) is entirely analogous, but with $H^{\bullet+2\dim X}_c(X;\bk)$ instead of $H^\bullet(X;\bk)$ and with induced maps on cohomology going in the reverse direction.
\end{proof}

\subsection{Proof of the theorem}

We consider the cartesian diagram 
\[
\vcenter{\xymatrix@R=0.6cm{
\cB \ar[r]^{i_\cB} \ar[d] \ar@{}[dr]|{\square} & \tcN \ar[d]^{\mu}\ar[r]^{i_{\tcN}} \ar@{}[dr]|{\square} & \tfg \ar[d]^{\pi} \\
\{0\} \ar[r]_{i_0} & \cN \ar[r]_{i_\cN} & \fg}}
\]
where again all maps are the natural ones.
By the base-change theorem, we obtain an isomorphism
\[
H^{\bullet-2N} \big( i_0^* \Spr(\bk) \big) \simto H^\bullet(\cB;\bk).
\]
This induces a map of endomorphism rings, which we write as
\begin{equation}\label{eq:sigma}
\sigma: \End(\Spr(\bk)) \to \End(H^\bullet(\cB;\bk)) \overset{\eqref{eqn:poincare}}{\cong} \End(H^\bullet_c(\cB;\bk)).
\end{equation}
Denote by $\tilde\rho(\bk), \tilde\varphi(\bk) : \End(\Groth(\bk)) \to \End(\Spr(\bk))$ the maps induced by $\tilde\rho^\bk$, $\tilde\varphi^\bk$.

\begin{lem}\label{lem:sigma-action}
We have $\gamma = \sigma \circ \tilde\rho(\bk)$ and $\gamma_c = \sigma \circ \tilde\varphi(\bk)$.
\end{lem}
\begin{proof}
The first equality follows from compatibility of the base-change isomorphism with the composition isomorphism for $(\cdot)^*$ functors, see \cite[Lemma B.7(a)]{ahr}. The second equality follows from the commutativity of the following diagram:
\[
\xymatrix@C=1.5cm@R=0.6cm{
j_0^* \bT_\fg \Groth(\bk) \ar[r]^-{\eqref{eqn:zero-section}} \ar[d]_-{\eqref{eqn:groth-fourier}} &  a_{\fg !} \Groth(\bk)[d] \ar[r]^-{\Co} &  a_{\cB !} t_! \ub \bk_{\tfg}[2d] \ar[d]^-{\eqref{eqn:cohc}} \\
j_0^* i_{\cN !} \Spr(\bk) \ar[r]^-{\BC} & i_0^* \Spr(\bk) \ar[r]^-{\BC} & a_{\cB!} \ub \bk_\cB[2N]
}
\]
To prove this commutativity, one must analyse the left vertical map, obtained by applying the functor $j_0^*$ to the isomorphism~\eqref{eqn:groth-fourier}. That is, one must recall from the proof of Lemma~\ref{lem:groth-fourier} the individual isomorphisms of which~\eqref{eqn:groth-fourier} is the composition, and determine what isomorphism each of these induces on the stalk at $0$. Apart from basic compatibilities between composition and base-change isomorphisms as in~\cite[Lemmas B.4 and B.7]{ahr}, the key results are Lemmas~\ref{lem:fourier-con-zero-section},~\ref{lem:fourier_!-zero-section-new} and~\ref{lem:fourier^*-zero-section-new}. Explicitly, a special case of Lemma~\ref{lem:fourier_!-zero-section-new} provides the commutativity of the diagram
\[
\xymatrix@C=1.5cm@R=0.6cm{
j_0^* \bT_\fg f_! x_!\ub\bk_{\tfg}[d] \ar[d]_-{\eqref{eqn:fourier_!}} \ar[r]^-{\eqref{eqn:zero-section}} 
& a_{\fg !}f_!x_!\ub\bk_{\tfg}[2d] \ar[r]^-{\Co}
& (a_{\fg\times\cB})_!x_!\ub\bk_{\tfg}[2d] \ar[d]^-{\Co}\\
j_0^* f_! \bT_{\fg\times\cB} x_!\ub\bk_{\tfg}[d] \ar[r]^-{\BC}
& a_{\cB !} (y')^* \bT_{\fg\times\cB} x_!\ub\bk_{\tfg}[d] \ar[r]^-{\eqref{eqn:zero-section}}
& a_{\cB !} p'_! x_!\ub\bk_{\tfg}[2d]
}
\]
where $p'$ denotes the vector bundle projection $\fg\times\cB\to\cB$ and $y':\cB\to\fg\times\cB$ denotes the embedding of the zero section in that vector bundle. 
Similarly, Lemma~\ref{lem:fourier^*-zero-section-new} provides the commutativity of the diagram
\[
\xymatrix@R=0.6cm{
a_{\cB !} (y')^* \bT_{\fg\times\cB} x_! \ub\bk_{\tfg}[d] \ar[rr]^-{\eqref{eqn:zero-section}} \ar[d]_-{\eqref{eqn:fourier^*}} &&
  a_{\cB !} p'_! x_! \ub\bk_{\tfg}[2d] \ar[d]^-{\Co} \\
a_{\cB !} (y')^* (\ut x)^* \bT_{\tfg} \ub\bk_{\tfg}[d+N] \ar[r]^-{\Co}  &
  a_{\cB !} y^* \bT_{\tfg} \ub\bk_{\tfg}[d+N] \ar[r]^-{\eqref{eqn:zero-section}} &
  a_{\cB !} t_! \ub\bk_{\tfg}[2d] 
}
\]
and Lemma~\ref{lem:fourier-con-zero-section} provides the commutativity of the diagram
\[
\xymatrix@C=2cm@R=0.6cm{
a_{\cB !} y^* \bT_{\tfg} \ub\bk_{\tfg} [d+N] \ar[r]^-{\eqref{eqn:zero-section}} \ar[d]_-{\eqref{eqn:fourier-con}} &  a_{\cB !} t_! \ub\bk_{\tfg}[2d] \ar[d]^-{\eqref{eqn:cohc}} \\
a_{\cB !} y^* y_! \ub\bk_{\cB} [2N]  \ar[r] & a_{\cB !} \ub\bk_{\cB}[2N]
}
\]
where the lower horizontal map is the obvious one.
Details are left to the reader. 
\end{proof}
 
\begin{thm}
\label{thm:main-easy}
Assume that $\bk$ embeds either in a $\Q$-algebra, or in an $\F_p$-algebra where $p$ does not divide $|\bW|$.
Then:
\begin{enumerate}
\item the map $\sigma : \End(\Spr(\bk)) \to \End(H^\bullet(\cB;\bk))$
  defined in \eqref{eq:sigma} is injective;
\item we have $\drho^\bk = \dvarphi^\bk \circ \sgn$.
\end{enumerate}
\end{thm}
\begin{proof}
By Lemmas \ref{lem:poincare} and \ref{lem:gamma-action} we have
$\gamma_c \circ \dr_\bk = \gamma \circ \dr_\bk \circ \sgn$. Then by
Lemma \ref{lem:sigma-action} we deduce that $\sigma \circ
\tilde\varphi(\bk) \circ \dr_\bk = \sigma \circ \tilde\rho(\bk) \circ
\dr_\bk \circ \sgn$. Moreover, by Lemma \ref{lem:gamma-action} again
and Lemma \ref{lem:cohomology-B}(2), these morphisms are injective. As
$\tilde\varphi(\bk) \circ \dr_\bk$ is an isomorphism, we deduce (1),
and also that $\tilde\varphi(\bk) \circ
\dr_\bk = \tilde\rho(\bk) \circ \dr_\bk \circ \sgn$. This proves (2),
as by definition we have $\drho^\bk=\tilde\rho(\bk) \circ
\dr_\bk$ and $\dvarphi^\bk = \tilde\varphi(\bk) \circ \dr_\bk$.
\end{proof}

\section{Consequences}
\label{sect:zeroweight}

Reverting to a more usual notion
of perverse sheaves, 
we let $\Perv_G(\cN,\bk)$ denote the category of $G$-equivariant perverse $\bk$-sheaves on $\cN$ (assumed to be constructible with respect to the stratification by $G$-orbits). Let $\fN_\bk$ be the set of isomorphism classes of simple objects of $\Perv_G(\cN,\bk)$. (Abusing notation slightly, we will sometimes write $\cA\in\fN_\bk$ when we mean that $\cA$ is such a simple object.) Let $\Rep(\bW,\bk)$ denote the category of finitely-generated $\bk[\bW]$-modules, and let $\Irr\bk[\bW]$ be the set of isomorphism classes of simple objects of $\Rep(\bW,\bk)$. (Again, we will sometimes write $E\in\Irr\bk[\bW]$ when we mean that $E$ is such a simple object.) Note that both categories $\Perv_G(\cN,\bk)$ and $\Rep(\bW,\bk)$ are noetherian, since $\bk$ is noetherian, $G$ has finitely many orbits on $\cN$, and $\bW$ is finite.

\subsection{Springer correspondence}

One of the main objects of study in~\cite{ahr} was the functor
\[
\Ss_{\drho}:\Perv_G(\cN,\bk)\to\Rep(\bW,\bk), \quad \cA\mapsto\Hom_{\Perv_G(\cN,\bk)}(\Spr(\bk),\cA),
\]
where the $\bW$-action on $\Hom(\Spr(\bk),\cA)$ is defined 
by $w \cdot f = f \circ
\drho(w^{-1})$. 
By~\cite[\S 7]{ahr}, the functor $\Ss_{\drho}$
commutes with restriction to Levi subgroups in an appropriate sense,
and 
it is exact (i.e.~$\Spr(\bk)$ is projective in $\Perv_G(\cN,\bk)$).
On the other hand, the functor $\Ss_{\dvarphi}$ defined in the same
way but using the map $\dvarphi$ was studied implicitly
in~\cite{juteau} (in its \'etale version using Fourier--Deligne
transform) and explicitly in~\cite{mautner} (in the same setting as
the present paper).
We conclude from Theorem~\ref{thm:main-intro} that 
$\Ss_{\drho}$ and $\Ss_{\dvarphi}$
differ only by a sign twist.
Thus all the results proved about either of them are valid for both.
From now on, 
we will use the notation $\Ss$ to refer to $\Ss_{\dvarphi}$. 

\begin{prop}  \label{prop:left-adjoint}
$\Ss$ has a left adjoint $T:\Rep(\bW,\bk)\to\Perv_G(\cN,\bk)$ such that the natural transformation $\mathrm{id}\to \Ss\circ T$ is an isomorphism. 
\end{prop}

\begin{proof}
The argument is almost the same as that in~\cite[pp.~403--404]{gabriel}. Let $M\in\Rep(\bW,\bk)$. Since $\Rep(\bW,\bk)$ is noetherian, there exists an exact sequence 
\begin{equation} \label{eqn:pres1}
\xymatrix{
\bk[\bW]^{\oplus m}\ar[r]^{\alpha} & \bk[\bW]^{\oplus n}\ar[r] & M\ar[r] & 0.
}
\end{equation}
Here $\alpha$ can be viewed as an $n\times m$ matrix of elements of $\bk[\bW]\simto\End_{\bk[\bW]}(\bk[\bW])$ (where the latter identification sends $w\in\bW$ to right multiplication by $w^{-1}$). Let $\dvarphi(\alpha)$ be the $n\times m$ matrix of elements of $\End(\Spr(\bk))$ obtained by applying $\dvarphi$ to each entry of $\alpha$. Define $T(M)\in\Perv_G(\cN,\bk)$ by the exact sequence 
\begin{equation} \label{eqn:pres2} 
\xymatrix{
\Spr(\bk)^{\oplus m}\ar[r]^{\dvarphi(\alpha)} &\Spr(\bk)^{\oplus n}\ar[r] & T(M)\ar[r] & 0.
}
\end{equation}
One can easily check that $T(M)$ represents the functor $\Hom_{\bk[\bW]}(M,\Ss(-))$. By a general principle~\cite[p.~346, Proposition 10]{gabriel}, there is a unique way to define $T$ on morphisms so that it becomes a functor that is left adjoint to $\Ss$.

Since $\dvarphi$ is a ring isomorphism, we have an isomorphism $\bk[\bW]\simto\Ss(\Spr(\bk))$ in $\Rep(\bW,\bk)$ sending $y\in\bW$ to $\dvarphi(y^{-1})\in\End(\Spr(\bk))=\Ss(\Spr(\bk))$. The induced ring isomorphism $\bk[\bW]\simto\End_{\bk[\bW]}(\bk[\bW])\simto\End_{\bk[\bW]}(\Ss(\Spr(\bk)))$ coincides with $\Ss\circ\dvarphi$.  
Applying the exact functor $\Ss$ to~\eqref{eqn:pres2} and comparing with~\eqref{eqn:pres1}, we conclude that the natural transformation $\mathrm{id}\to \Ss\circ T$ is an isomorphism.
\end{proof}

\begin{cor} \label{cor:quotient}
$\Ss$ is a quotient functor. That is, it induces an equivalence of categories between $\Rep(\bW,\bk)$ and the quotient of $\Perv_G(\cN,\bk)$ by the kernel of $\Ss$.
\end{cor}

\begin{proof}
This follows from Proposition~\ref{prop:left-adjoint} as in~\cite[p.~374, Proposition 5]{gabriel}.
\end{proof}

\begin{cor}[Springer correspondence over $\bk$] \label{cor:correspondence}
There is a subset $\fN_\bk^0\subseteq\fN_\bk$ such that $\Ss$ induces a bijection $\fN_\bk^0\simto\Irr\bk[\bW]$, and $\Ss(\cA)=0$ for $\cA\in\fN_\bk\setminus\fN_\bk^0$.
\end{cor} 
\begin{proof}
The only part which is not immediate from Corollary~\ref{cor:quotient} is the statement that any $E\in\Irr\bk[\bW]$ arises as $\Ss(\cA)$ for some $\cA\in\fN_\bk$. But any such $E$ is a quotient of $\bk[\bW]$ in $\Rep(\bW,\bk)$, hence (since $T$ is right exact) $T(E)$ is a quotient of $T(\bk[\bW])\cong\Spr(\bk)$. Any simple quotient $\cA$ of $T(E)$ is therefore also a quotient of $\Spr(\bk)$, 
hence $\Ss(\cA)$ is nonzero. Since $\Ss$ is exact and $\Ss(T(E))\cong E$ is simple, we conclude that $\Ss(\cA)\cong E$ as required. This shows, incidentally, that any maximal subobject of $T(E)$ belongs to the kernel of $\Ss$, so $T(E)$ has a unique maximal subobject and 
$\cA$ is the head of $T(E)$.
\end{proof}

We can give a formula for the inverse bijection $\Irr\bk[\bW]\simto\fN_\bk^0$. Recall that we have a
surjective group homomorphism $\pi_1(\fg_\rs)\to\bW$, so every
$E\in\Irr\bk[\bW]$ determines a simple $\bk$-local system $\cL_E$ on
$\fg_\rs$. (Here we have to choose a base point in $\fg_\rs$, but this choice is irrelevant for the next proposition.)

\begin{prop} \label{prop:explicit}
For any $E\in\Irr\bk[\bW]$, the element of $\fN_\bk^0$ corresponding
to $E$ is $\bT_\fg(j_\rs)_{!*}\cL_{E}[d]$
\textup{(}or, strictly speaking, the restriction to $\cN$ of this
simple perverse sheaf supported on $\cN$\textup{)}.
\end{prop}

\begin{proof}
The argument is the same as 
in the \'etale setting (see~\cite[\S 6.2.4]{juteau}).
Namely, since $\pi_\rs$ is a Galois covering map with group $\bW$, we have an isomorphism
\begin{equation}
\Hom_{\Sh(\fg_\rs,\bk)}(\pi_{\rs *}\ub\bk,\cL_E)\cong \Hom_{\bk[\bW]}(\bk[\bW],E)\cong E
\end{equation}
of $\bk[\bW]$-modules, where the action of $w\in\bW$ on the middle (resp.~left) term is by precomposing with right multiplication by $w$ (resp.~with ${\tilde r}_{\bk}(w^{-1})$). Since $(j_\rs)_{!*}$ is fully faithful, this induces an isomorphism
\begin{equation}
\Hom_{\Perv(\fg,\bk)}(\Groth(\bk),(j_\rs)_{!*}\cL_E[d])\cong E
\end{equation}
of $\bk[\bW]$-modules. Applying the equivalence $\bT_\fg$ and using Lemma~\ref{lem:groth-fourier}, we deduce an isomorphism of $\bk[\bW]$-modules
\begin{equation}
\Hom_{\Perv_G(\cN,\bk)}(\Spr(\bk),\bT_\fg(j_\rs)_{!*}\cL_E[d])\cong E,
\end{equation}
where $\bW$ acts on $\Spr(\bk)$ via $\dvarphi$. The claim follows.
\end{proof}

\begin{rmk}
Similar descriptions of the left and right adjoints of $\Ss$ in the case that $\bk$ is a field are given by~\cite[Theorem 6.2 and Proposition 7.2]{mautner}. 
\end{rmk}

Previous definitions of the Springer correspondence have required $\bk$ to be a field. 
When $\bk=\Q_\ell$, the above correspondence is Springer's original
version \cite{springer}, differing by a sign twist from Lusztig's reformulation
\cite{lus,bm}. This follows from Lemma \ref{lem:gamma-action} and
Theorem \ref{thm:main-easy}(1).
This `ordinary' Springer correspondence has been determined explicitly,
see~\cite[Section 13.3]{carter} (where Lusztig's version is used).

When $\bk$ is a finite field of characteristic $\ell>0$, there is also
the `modular' Springer correspondence of~\cite[Theorem
  6.2.8]{juteau}. It is given by the same formula as
Proposition~\ref{prop:explicit}, but in the \'etale setting with
Fourier--Deligne transform. The change
from Fourier--Deligne transform to Fourier--Sato transform makes no
difference to the combinatorics of the correspondence, with the
obvious identification of the two versions of $\fN_\bk$.
Indeed, by \cite[Section 3]{jls}, the modular Springer correspondence
is completely determined by the ordinary Springer correspondence and
the decomposition matrix of $\bW$, and the results used to prove this, such
as \cite[Theorem 6.3.2]{juteau}, are equally valid for the Fourier-Sato
transform. (The fact that the decomposition matrix is unknown in
general is irrelevant for this uniqueness argument.) 

The modular Springer correspondence has been explicitly
determined in \cite{juteau} for types $A_n$ and $G_2$, in
\cite{jls} for classical types when $\ell \neq 2$, and in as-yet
unpublished tables for the remaining exceptional types.

\subsection{Zero weight spaces of small representations}

Assume now that $\bk$ is a field and that $G$ is simple and simply-connected. Let $\Gv$ denote the split connected reductive group over $\bk$ whose root datum is dual to that of $G$ (so that $\Gv$ is simple of adjoint type). For any dominant coweight $\lambda$ of $G$, there is an irreducible representation $L(\lambda)$ of $\Gv$ with highest weight $\lambda$. The zero weight space $L(\lambda)_0$ carries a representation of the Weyl group of $\Gv$, which can be identified with $\bW$.

The main result of~\cite{ahr} determined the representation of $\bW$
on $L(\lambda)_0$ in the case where $\lambda$ is \emph{small} (meaning
that the convex hull of $\bW\lambda$ does not contain
$2\check{\alpha}$ for any root $\alpha$). Namely, we have the
following isomorphism, obtained from~\cite[(1.3)]{ahr} by tensoring
both sides by the sign character (recall that in the present paper,
$\Ss$ means $\Ss_{\dvarphi}$ whereas in \cite{ahr} it means $\Ss_{\drho}$):
\begin{equation} \label{eqn:zerowtspace}
L(\lambda)_0\cong \Ss(\pi_*(\IC(\Gr^\lambda,\bk)|_{\cM})).
\end{equation}  
Here $\Gr^\lambda$ is the orbit in the affine Grassmannian of $G$ labelled by $\lambda$, $\cM$ is the $G$-stable locally closed subvariety of the affine Grassmannian defined in~\cite{ah} (which intersects $\overline{\Gr^\lambda}$ in an open dense subvariety for all small $\lambda$), and $\pi:\cM\to\cN$ is the $G$-equivariant finite map defined and described in~\cite{ah}. (We follow the notation of~\cite{ah} for this finite map since we no longer need the Grothendieck--Springer map denoted $\pi$ in previous sections.)

We can now express the right-hand side of~\eqref{eqn:zerowtspace} in terms of the Springer correspondence of Corollary~\ref{cor:correspondence}. The nilpotent orbits contained in $\pi(\Gr^\lambda\cap\cM)$ for each small $\lambda$ are described explicitly in~\cite[Table 1, Table 6]{ah}, and we have the following dichotomy~\cite[Proposition 3.2]{ah}:
\begin{prop} \label{prop:dichotomy}
For any small coweight $\lambda$ of $G$, one of the following occurs.
\begin{enumerate}
\item $\pi(\Gr^\lambda\cap\cM)$ is a single $G$-orbit $\cO_\lambda$, and $\pi$ restricts to an isomorphism $\Gr^\lambda\cap\cM\simto\cO_\lambda$.
\item $\pi(\Gr^\lambda\cap\cM)$ is the union of two $G$-orbits, of which one, say $\cO_\lambda$, contains the other in its closure, and $\pi$ restricts to a non-trivial $2:1$ Galois covering $U_\lambda\to\cO_\lambda$, where $U_\lambda$ is open in $\Gr^\lambda\cap\cM$.
\end{enumerate}
\end{prop}
\begin{cor} \label{cor:dichotomy}
In the two cases of Proposition~{\rm \ref{prop:dichotomy}}, we have, respectively:
\begin{enumerate}
\item $L(\lambda)_0\cong\Ss(\IC(\cO_\lambda,\bk))$. Hence $L(\lambda)_0$ is either simple or zero.
\item Assuming that the characteristic of $\bk$ is not $2$, $L(\lambda)_0\cong \Ss(\IC(\cO_\lambda,\bk))\oplus\Ss(\IC(\cO_\lambda,\cL))$ where $\cL$ denotes the nontrivial rank-one $G$-equivariant $\bk$-local system on $\cO_\lambda$ arising from the Galois covering $U_\lambda\to\cO_\lambda$. Hence $L(\lambda)_0$ is either the direct sum of two nonisomorphic simples, simple, or zero.
\end{enumerate}
\end{cor}

\begin{proof}
In the first case, the finiteness of $\pi$ and Proposition~\ref{prop:dichotomy}(1) imply that
\begin{equation*} 
\pi_*(\IC(\Gr^\lambda,\bk)|_{\cM})\cong\IC(\cO_\lambda,\bk).
\end{equation*}
In the second case, the finiteness of $\pi$, the assumption that the characteristic of $\bk$ is not $2$, and Proposition~\ref{prop:dichotomy}(2) imply that
\begin{equation*} 
\pi_*(\IC(\Gr^\lambda,\bk)|_{\cM})\cong\IC(\cO_\lambda,\bk)\oplus\IC(\cO_\lambda,\cL).
\end{equation*}
The result now follows from~\eqref{eqn:zerowtspace} and Corollary~\ref{cor:correspondence}.
\end{proof}

\begin{rmk}
In the situation of Corollary~\ref{cor:dichotomy}(2), the
$G$-equivariant fundamental group of $\cO_\lambda$ is either $\Z/2\Z$
or $S_3$ (the latter occurring only in the case of
Example~\ref{ex:g2} below). Hence $\cL$ is in fact the unique nontrivial
rank-one $G$-equivariant local system on $\cO_\lambda$.
\end{rmk}

If $\mathrm{char}(\bk)=0$,
Corollary \ref{cor:dichotomy}
recovers the calculations in \cite{R1,R2}, as 
noted in \cite{ahr}.
Now we consider the case where 
$\mathrm{char}(\bk)=\ell > 0$.

\begin{ex}
If $G = SL_n$, we are always in case $(1)$. Let $\lambda = (a_1, \ldots, a_n)$
where $a_1 \geq \ldots \geq a_n \geq -1$ and $\sum_{i = 1}^n a_i = 0$.
Let $\hat \lambda = (a_1 + 1, \ldots, a_n + 1)$ be the associated
partition of $n$. Then the orbit denoted $\cO_\lambda$ above
consists of nilpotent matrices with Jordan type $\hat\lambda$.
Using the determination of
the modular Springer correspondence in this case \cite[Theorem
  6.4.1]{juteau}, we recover the well-known fact \cite[A.23(5)]{jantzen1}
that $L(\lambda)_0 \simeq \DB^{{\hat \lambda}^t}$ when
$\lambda$ is $\ell$-restricted, and $0$ otherwise. (We use the standard notation for the simple
modules of $\bk[S_n]$.) 
\end{ex}

\begin{ex}
Suppose $G$ is of type $B$, $C$ or $D$ and
$\ell > 2$. The modular Springer correspondence for this case has been determined in
\cite[Theorem 4.21]{jls}. Combining that result with \cite[Table
  1]{ah}, we obtain the explicit description of zero weight spaces
of small representations displayed in Table \ref{table}. We use the
notation of \cite{ah} for the small coweights, and the
notation of \cite{jls} for the simple $\bk[\bW]$-modules.
If $\ell = 3$ and $G$ is of type $B_3$, one should interpret
$\DB^{((1^3),\emptyset)}$ as $0$. Among all the cases considered in
the present example, this is the only time that one of the
intersection cohomology complexes appearing in Corollary
\ref{cor:dichotomy} is killed by $\Ss$.
\end{ex}

\begin{table}[hbt]
\[
\begin{array}{|c|cc|c|}
\hline
G
& \lambda
&
& L(\lambda)_0
\\
\hline
B_n
& (2 1^{2j} 0^{n - 2j - 1})
& (\text{for } 1\leq j \leq \frac{n - 1}{2})
& \DB^{((n - j - 1, j), (1))}
\oplus\ \DB^{((n - j - 1, j, 1), \emptyset)}\\
\cline{2-4}
& (2 0^{n - 1})
&
& \DB^{((n - 1), (1))}
\\
\cline{2-4}
& (1^{2j} 0^{2n - j})
&
& \DB^{((n - j, j), \emptyset)}
\\
\hline
C_n
& (1^j 0^{n - j})
&
& \DB^{((n - \frac j 2), (\frac j 2))} \text{ if } j \text{ even}\\
&&& \DB^{((\frac{j-1}2), (n - \frac{j-1}2))} \text{ if } j \text{ odd}
\\
\hline
D_n
& (2 1^{n - 2} (\pm 1))
& (\text{if } n \text{ odd})
& \EB^{[(\frac{n-1}2, 1), (\frac{n-1}2)]}
\\
\cline{2-4}
& (2 1^{2j} 0^{n - 2j - 1})
& (\text{for } 1\leq j < \frac{n - 1}{2})
& \EB^{[(n - j - 1, 1), (j)]}\ \oplus
\EB^{[(n - j - 1), (j, 1)]}
\\
\cline{2-4}
& (2 0^{n - 1})
&
& \EB^{[(n - 1, 1), \emptyset]}
\\
\cline{2-4}
& (1^{n - 1} (\pm 1))
& (\text{if } n \text{ even})
& \EB^{[(\frac n 2), \pm]}
\\
\cline{2-4}
& (1^{2j} 0^{n - 2j})
& (\text{for } 0\leq j < \frac{n}{2})
& \EB^{[(n - j), (j)]}
\\
\hline
\end{array}
\]
\caption{Zero weight spaces of small representations for classical groups}\label{table}
\end{table}

\begin{ex} \label{ex:g2}
Assume $G$ is of type $G_2$ and $\ell > 2$. The modular Springer
correspondence for this type is described in~\cite[\S 7.7]{juteau}. 
Let $\lambda$ be the higher fundamental coweight, so
that $L(\lambda)$ is either the adjoint representation of $\Gv$ or its
unique simple quotient if $\ell=3$. We are in case (2),
$\cO_\lambda$ is the subregular orbit,
$\Ss(\IC(\cO_\lambda,\cL))=0$, and accordingly
$L(\lambda)_0\cong\Ss(\IC(\cO_\lambda,\bk))$
(of dimension $1$ if $\ell = 3$ and $2$
if $\ell > 3$).
\end{ex}

\begin{ex}
Among the cases where $G$ is of exceptional type and $\ell > 2$, there
are only two occasions other than
that described in the preceding example where $\Ss$
kills one of the intersection cohomology complexes appearing in
Corollary \ref{cor:dichotomy}. Namely, if
$\lambda = 3\wt_1$ or $3\wt_6$ in type $E_6$ (in the numbering of
\cite{bourbaki}), Corollary \ref{cor:dichotomy}(1) says that
$L(\lambda)_0  \simeq \Ss(\IC(2A_2,\bk))$.
If $\ell = 3$ the latter vanishes; in this case $L(\lambda)$ is the
Frobenius twist of a minuscule representation.
\end{ex}

Since there are only finitely many small coweights in the exceptional
types, one can compute in every case the image of $L(\lambda)_0$ in
the Grothendieck group of $\Rep(\bW,\bk)$, using the known
answer in characteristic zero \cite{R1,R2} together with the known
decomposition matrices for the small representations of $\Gv$
\cite{LUB} and the decomposition matrices for the exceptional Weyl groups
which are available in GAP3.
(The latter decomposition matrices were first
computed in \cite{Khos} for $\ell = 3$ and \cite{KM} for $\ell > 3$.) 
As predicted by Corollary \ref{cor:dichotomy}, when $\ell > 2$
one finds that $L(\lambda)_0$ is multiplicity-free and has at most two
simple constituents. Note that
in the case where two simple constituents appear, Corollary
\ref{cor:dichotomy}(2) gives the additional information that
$L(\lambda)_0$ is the direct sum of the two.

If we are in the situation of Proposition
\ref{prop:dichotomy}(2) and $\ell = 2$, it is possible for
$L(\lambda)_0$ to have length greater than $2$. (For example,
in type $E_7$ the $\bk[\bW]$-module $L(\wt_3)_0$ has four distinct simple
constituents, three of which appear with multiplicity $2$, and the
last one with multiplicity $1$.)
In geometric terms, the reason that
characteristic $2$ is different is that the sheaf
$(\pi|_{U_\lambda})_*\ub\bk_{U_\lambda}$ is a non-split extension of
$\ub\bk_{\cO_\lambda}$ by itself. Since the intermediate
extension functor preserves socles and heads~\cite[Proposition
  2.28]{juteau-aif}, one can conclude that
$\pi_*(\IC(\Gr^\lambda,\bk)|_{\cM})$ has socle and head equal to
$\IC(\cO_\lambda,\bk)$. However, it may also have additional simple
constituents in the middle, supported on
$\overline{\cO_\lambda}\setminus\cO_\lambda$. When one then applies the functor $\Ss$ to obtain
$L(\lambda)_0$,
diverse situations occur: there are cases where the head and socle are
killed by $\Ss$ and the constituents in the middle remain, and others where only the head and
socle survive.

\section{\'Etale version}
\label{sect:etale}

The arguments of this paper can be adapted to prove an analogue of the main result in the \'etale topology over a finite field, substituting the Fourier--Deligne transform for the Fourier--Sato transform. In fact, the `change of rings' operation is simpler (since it is at the heart of the construction of the derived categories); in particular we don't need the functorial point of view. In this section, we describe the set-up, and we briefly indicate some of the points that require special attention.

\subsection{Derived categories of sheaves}

Fix a finite field $\F_q$ of characteristic $p$, and a prime number $\ell \ne p$.  Let $X$ be a variety over $\F_q$.

We say that a ring $\bk$ is \emph{admissible} if it is isomorphic to a finite integral extension of some $\Z/\ell^n\Z$ or of $\Z_\ell$, or else to a (possibly infinite) integral extension of $\Z_\ell$ or $\Q_\ell$.  For each of these rings, there is a triangulated category $\Dbc(X,\bk)$, to be thought of as the category of `constructible complexes of $\bk$-sheaves'.  The definition of $\Dbc(X,\bk)$ is rather complicated when $\bk$ is not a finite ring; see~\cite[Appendix~A]{kw} for an exposition.  One special case will be recalled in Section~\ref{subsect:changerings-et}.
In each case, $\Dbc(X,\bk)$ has a `natural' $t$-structure and a perverse $t$-structure, with hearts 
$\Shc(X,\bk)$ and $\Perv(X,\bk)$, respectively.  
We have a constant sheaf $\ub \bk_X \in \Shc(X,\bk)$.

Suppose now that $\psi: \F_q \to \bk^\times$ is a nontrivial homomorphism. If $E$ and $E^*$ are dual vector bundles over an $\F_q$-variety $Y$, there is a \emph{Fourier--Deligne transform}, denoted $\bT^\psi_E: \Dbc(E,\bk) \simto \Dbc(E^*,\bk)$.  Its theory is developed in the generality we need in~\cite[Chap.~5]{juteau}.  

\subsection{Springer and Grothendieck functors}

Let $G$ be a split connected reductive group over $\F_q$.  The varieties $\cB$, $\cN$, $\tcN$, etc., are then also defined over $\F_q$, and we can define objects
\[
\Groth(\bk) \in \Perv(\fg,\bk), \qquad \Spr(\bk) \in \Perv(\cN,\bk).
\] 
We also assume that $\fg$ can be equipped with a $G$-invariant nondegenerate bilinear form, and we fix such a form.  (This assumption imposes a constraint on $p$; see~\cite[\S 2.5]{letellier}.)  Having identified $\fg \cong \fg^*$, we regard $\bT^\psi_\fg$ as an autoequivalence
\[
\bT^\psi_\fg: \Dbc(\fg,\bk) \simto \Dbc(\fg,\bk).
\]
From~\cite[\S6.2]{juteau}, we have \'etale analogues of Lemmas~\ref{lem:groth-res} and~\ref{lem:groth-fourier}, although the latter now includes a Tate twist:
\begin{align*}
i_\cN^* \Groth(\bk)[-r] &\cong \Spr(\bk), &
\bT^\psi_\fg \circ \Groth(\bk) &\cong i_{\cN!} \circ \Spr(\bk)(-N-r).
\end{align*}
One can then define analogues of $\drho^\bk$ and $\dvarphi^\bk$, see~\cite[Theorem~6.2.1, Corollary~6.2.5]{juteau}.

\subsection{Change of rings}
\label{subsect:changerings-et}

Let us assume that $\bk$ is a finite integral extension of $\Z_\ell$, with maximal ideal $\fm$.  Any quotient $\bk/\fm^n$ is a finite ring.  For $n \ge 1$, we consider the full triangulated subcategory $\Dbctf(X,\bk/\fm^n) \subset \Dbc(X,\bk/\fm^n)$ whose objects are quasiisomorphic to a bounded complex of flat sheaves of $\bk/\fm^n$-modules.  Then
\begin{equation}\label{eqn:inv-limit}
\Dbc(X,\bk) = 
\varprojlim_n
\Dbctf(X,\bk/\fm^n)
\end{equation}
by definition. If $\bk'$ is any quotient of $\bk$, we define the functor
\[
\Ex^X_{\bk,\bk'}: \Dbc(X,\bk) \to \Dbc(X,\bk')
\]
to be the composition of the canonical functor $\Dbc(X,\bk) \to \Dbctf(X,\bk')$ with the inclusion $\Dbctf(X,\bk') \to \Dbc(X,\bk')$. This functor is an `extension of scalars' functor rather than a forgetful functor as in \S\ref{subsect:changerings}. Sheaf functors $(\cdot)_!$, $(\cdot)^*$ and the tensor product on $\Dbc(\cdot,\bk)$, as well as the usual isomorphisms between them (composition, base-change, projection formula) are defined termwise with respect to the inverse limit~\eqref{eqn:inv-limit}, so that they are compatible with the functors $\Ex_{\bk,\bk'}$ in the natural sense. Similar remarks apply to the Fourier--Deligne transform (which is defined in terms of the preceding operations) and its compatibility properties similar to \eqref{eqn:fourier^*}, \eqref{eqn:fourier_!}. Moreover, one clearly has $\Ex^X_{\bk,\bk'}(\ub \bk_X) = \ub {\bk'}_X$.

The following result follows from these remarks:

\begin{prop}
\label{prop:scalars-et}
Suppose that $\bk$ is a finite integral extension of $\Z_\ell$, and that $\bk'$ is a quotient of $\bk$.
Then there is a canonical $\bk$-algebra homomorphism
\[
\End(\Spr(\bk)) \to \End(\Spr({\bk'}))
\]
induced by the functor $\Ex^\cN_{\bk,\bk'}$. Moreover,
the following diagrams commute:
\[
\vcenter{\xymatrix@R=0.6cm{
\bk[\bW] \ar[d]\ar[r]^-{\drho^\bk} & \End(\Spr(\bk)) \ar[d] \\
\bk'[\bW] \ar[r]^-{\drho^{\bk'}} & \End(\Spr({\bk'}))}}
\qquad\qquad
\vcenter{\xymatrix@R=0.6cm{
\bk[\bW] \ar[d]\ar[r]^-{\dvarphi^\bk} & \End(\Spr(\bk)) \ar[d] \\
\bk'[\bW] \ar[r]^-{\dvarphi^{\bk'}} & \End(\Spr({\bk'}))}}
\]
\end{prop}

\subsection{Main result}

The following statement is analogous to Theorem~\ref{thm:main}.

\begin{thm}\label{thm:main-et}

Let $\bk$ be an admissible ring.
The two maps $\drho^\bk, \dvarphi^\bk: \bk[\bW] \to \End(\Spr(\bk))$ are related as follows:
$\drho^\bk = \dvarphi^\bk \circ \sgn$.
\end{thm}
\begin{proof}
In view of Proposition~\ref{prop:scalars-et}, it suffices to treat the cases where $\bk$ is either an integral extension of $\Z_\ell$ or an algebraic extension of $\Q_\ell$.  Those cases are covered by the classical arguments explained in Section~\ref{sect:thm-Q}.
\end{proof}

\appendix
\section{Isomorphisms of sheaf functors}
\label{sect:isomsheaf}

In this appendix, we collect a few lemmas on compatibility of various isomorphisms of sheaf functors with change of rings, and compatibility between isomorphisms arising from Fourier--Sato transform. Throughout, $\bk'$ denotes a $\bk$-algebra.

\subsection{Composition and change of rings}

Suppose we have two maps
\begin{equation}\label{eqn:comp-setting}
X \xrightarrow{f} Y \xrightarrow{g} Z
\end{equation}
of topological spaces, and let $h = g \circ f$.

\begin{lem}\label{lem:comp-for}
In the setting of~\eqref{eqn:comp-setting}, the following diagrams commute:
\[
\xymatrix@R=0.6cm{
\For^Z_{\bk,\bk'} g_! f_! \ar[r]^{\eqref{eqn:for_!}} \ar[dd]_{\Co} &
  g_! \For^Y_{\bk,\bk'} f_! \ar[d]^{\eqref{eqn:for_!}} \\
&  g_! f_! \For^X_{\bk,\bk'} \ar[d]^{\Co} \\
\For^Z_{\bk,\bk'} h_! \ar[r]_{\eqref{eqn:for_!}} &
  h_! \For^X_{\bk,\bk'}}
\qquad
\xymatrix@R=0.6cm{
\For^X_{\bk,\bk'} f^* g^* \ar[r]^{\eqref{eqn:for^*}} \ar[dd]_{\Co} &
  f^* \For^Y_{\bk,\bk'} g^* \ar[d]^{\eqref{eqn:for^*}} \\
  & f^* g^* \For^Z_{\bk,\bk'} \ar[d]^{\Co} \\
\For^X_{\bk,\bk'} h^* \ar[r]_{\eqref{eqn:for^*}} &
  h^* \For^Z_{\bk,\bk'}}
\]
\end{lem}
\begin{proof}
By general principles, it suffices to check the corresponding statements for nonderived functors of abelian categories, and those are straightforward.
\end{proof}

\subsection{Base change and change of rings}

Consider a cartesian square of topological spaces:
\begin{equation}\label{eqn:bc-setting}
\vcenter{\xymatrix@R=0.6cm{
W \ar[r]^{g'} \ar[d]_{f'} \ar@{}[dr]|{\square} & X \ar[d]^{f} \\
Y \ar[r]_g & Z}}
\end{equation}

\begin{lem}\label{lem:bc-for}
In the setting of~\eqref{eqn:bc-setting}, the following diagram commutes:
\[
\xymatrix@R=0.6cm{
\For^Y_{\bk,\bk'} g^* f_! \ar[r]^{\eqref{eqn:for^*}} \ar[d]_{\BC} &
  g^* \For^Z_{\bk,\bk'} f_! \ar[r]^{\eqref{eqn:for_!}} &
  g^* f_! \For^X_{\bk,\bk'} \ar[d]^{\BC} \\
\For^Y_{\bk,\bk'} f'_! g'{}^* \ar[r]^{\eqref{eqn:for_!}} &
  f'_! \For^W_{\bk,\bk'} g'{}^* \ar[r]^{\eqref{eqn:for^*}} &
  f'_! g'{}^* \For^X_{\bk,\bk'} }
\]
\end{lem}
\begin{proof}
As above, it suffices to prove the analogous nonderived statement.  The reader is referred to~\cite[Proposition~2.5.11]{kas} for an explicit description of the nonderived base-change isomorphism.  It is easily seen that the construction there 
commutes with forgetful functors in the desired way.
\end{proof}

\subsection{Constant sheaf and change of rings}

Let $f : X \to Y$ be a continuous map of topological spaces. The following lemma is obvious.

\begin{lem}\label{lem:const-for}
The following diagram commutes:
\[
\xymatrix@R=0.6cm{
\For^X_{\bk,\bk'} f^* \ubb_Y \ar[r]^{\eqref{eqn:for^*}} \ar[d]_{\Const} &
  f^* \For^Y_{\bk,\bk'} \ubb_Y \ar[r]^{\eqref{eqn:for-constant}} &
  f^* \ubb_Y \Formod_{\bk,\bk'} \ar[d]^{\Const} \\
\For^X_{\bk,\bk'} \ubb_X \ar[rr]^-{\eqref{eqn:for-constant}} &
  & \ubb_X \Formod_{\bk,\bk'} }
\]
\end{lem}

\subsection{Fourier transform and change of rings}

Consider the context of \eqref{eqn:fourier_!}. This isomorphism is deduced from base-change and composition isomorphisms using the following diagram 
where $Q'$ is to $E'$ what $Q$ is to $E$:
\[
\xymatrix@R=0.6cm{
E' \ar[d]_-{f} \ar@{}[dr]|{\square} & Q' \ar[r]^-{\check q'} \ar[l]_-{q'} \ar[d] \ar@{}[dr]|{\square} & (E')^* \ar[d]^-{f'} \\
E & Q \ar[r]^-{\check q} \ar[l]_-{q} & E^*
}
\]

\begin{lem}\label{lem:fourier-for_!}
In the context of \eqref{eqn:fourier_!}, the following diagram commutes:
\[
\xymatrix@R=0.6cm{
\For^{E^*}_{\bk,\bk'} \bT_E f_! \ar[r]^{\eqref{eqn:for-fourier}} \ar[d]_{\eqref{eqn:fourier_!}}
  & \bT_E \For^E_{\bk,\bk'} f_! \ar[r]^{\eqref{eqn:for_!}}
  & \bT_E f_! \For^E_{\bk,\bk'} \ar[d]^{\eqref{eqn:fourier_!}} \\
\For^{E^*}_{\bk,\bk'} f'_! \bT_{E'} \ar[r]^{\eqref{eqn:for_!}}
  & f'_! \For^{E^*}_{\bk,\bk'} \bT_{E'} \ar[r]^{\eqref{eqn:for-fourier}}
  & f'_! \bT_{E'} \For^E_{\bk,\bk'}}
\]
\end{lem}

Now, consider the context of \eqref{eqn:fourier^*}. This isomorphism is deduced from base-change and composition isomorphisms using the following diagram, where both squares are cartesian and both triangles are commutative:
\[
\xymatrix@R=0.6cm{
E_1^* \ar@{}[dr]|{\square} & E_2^* \ar[l]_-{\ut \phi} & \\
Q_1 \ar[rd]_-{q_1} \ar[u]^-{\check q_1} & Q_{12} \ar@{}[dr]|{\square} \ar[l] \ar[r] \ar[u] \ar[d] & Q_2 \ar[lu]_-{\check q_2} \ar[d]^-{q_2} \\
& E_1 \ar[r]^-{\phi} & E_2
}
\]

\begin{lem}\label{lem:fourier-for^*}
In the context of \eqref{eqn:fourier^*}, the following diagram commutes:
\[
\xymatrix@R=0.6cm{
\For^{E_2^*}_{\bk,\bk'} \ut \phi^* \bT_{E_1} [-\rk E_1] \ar[r]^{\eqref{eqn:for^*}} \ar[d]_{\eqref{eqn:fourier^*}} &
  \ut \phi^* \For^{E_1^*}_{\bk,\bk'} \bT_{E_1} [-\rk E_1] \ar[r]^{\eqref{eqn:for-fourier}} &
  \ut \phi^* \bT_{E_1} \For^{E_1}_{\bk,\bk'} [-\rk E_1] \ar[d]^{\eqref{eqn:fourier^*}} \\
\For^{E_2^*}_{\bk,\bk} \bT_{E_2} \phi_! [-\rk E_2] \ar[r]^{\eqref{eqn:for-fourier}} &
  \bT_{E_2} \For^{E_2}_{\bk,\bk'} \phi_! [-\rk E_2] \ar[r]^{\eqref{eqn:for_!}} &
  \bT_{E_2} \phi_! \For^{E_1}_{\bk,\bk'} [-\rk E_2]}
\]
\end{lem}

\begin{proof}[Proofs of Lemmas~\ref{lem:fourier-for_!} and~\ref{lem:fourier-for^*}]
The claims follow from Lemmas~\ref{lem:comp-for} and~\ref{lem:bc-for}.
\end{proof}

\subsection{Fourier transform of the constant sheaf}
\label{ss:fourier-constant}

Consider the context of \eqref{eqn:fourier-con}.
By definition, we have $\bT_E(\ub M_E) = \check q_! q^*\ub M_E[\rk E] = \check q_! \ub M_Q[\rk E]$.   Let $U = \check q^{-1}(E^* \smallsetminus Y)$, and let $h: U \to Q$ and $i: Q \smallsetminus U \to Q$ be the inclusion maps.  It is easy to see that the complement $Q \smallsetminus U$ can be identified with $E$, and $\check q \circ i$ can be identified with $y \circ p: E \to Y \to E^*$.  Form the natural distinguished triangle
\[
h_! \ub M_U \to \ub M_Q \to i_! \ub M_E \to.
\]
It is easily checked that 
$\check q_!h_!(\ub M_U) = 0$.  Thus, the adjunction morphism
\begin{equation}\label{eqn:fourier-con-adj}
\ub M_Q \to i_! i^* \ub M_Q = i_! \ub M_E
\end{equation}
becomes an isomorphism after applying the functor $\check q_!$.  Next, we have a pair of composition isomorphisms
\begin{equation}\label{eqn:fourier-con-comp}
\check q_! i_! \ub M_E \simto (\check q \circ i)_! \ub M_E \simto y_!p_! \ub M_E.
\end{equation}
Finally, as explained in Lemma \ref{lem:cohc} below there is a canonical isomorphism
\begin{equation}\label{eqn:fourier-con-vec}
p_!(\ub M_E) \cong \ub M_Y[-2\rk E].
\end{equation}
The isomorphism~\eqref{eqn:fourier-con} is obtained by composing~\eqref{eqn:fourier-con-adj},~\eqref{eqn:fourier-con-comp}, and~\eqref{eqn:fourier-con-vec}.

\begin{lem}\label{lem:fourier-constant}
The following diagram commutes:
\[
\xymatrix@R=0.6cm{
\For^{E^*}_{\bk,\bk'} \bT_E\ubb_E \ar[r]^{\eqref{eqn:for-fourier}} \ar[d]_{\eqref{eqn:fourier-con}} &
  \bT_E \For^E_{\bk,\bk'} \ubb_E \ar[r]^{\eqref{eqn:for-constant}} &
  \bT_E \ubb_E \Formod_{\bk,\bk'} \ar[d]^{\eqref{eqn:fourier-con}} \\
\For^{E^*}_{\bk,\bk'} y_!\ubb_Y [-\rk E] \ar[r]^{\eqref{eqn:for_!}} &
  y_! \For^Y_{\bk,\bk'} \ubb_Y [-\rk E] \ar[r]^{\eqref{eqn:for-constant}} &
  y_! \ubb_Y \Formod_{\bk,\bk'} [-\rk E]
}
\]
\end{lem}

\begin{proof}

To show that~\eqref{eqn:fourier-con} commutes with $\For_{\bk,\bk'}$, we must show the same statement for each of~\eqref{eqn:fourier-con-adj},~\eqref{eqn:fourier-con-comp}, and~\eqref{eqn:fourier-con-vec}.  For~\eqref{eqn:fourier-con-comp}, this is Lemma~\ref{lem:comp-for}.  For~\eqref{eqn:fourier-con-adj}, this follows from the commutativity of the diagram
\[
\xymatrix@R=0.6cm{
\For^Q_{\bk,\bk'} \ar[r]\ar[d]
  & \For^Q_{\bk,\bk'} i_!i^* \ar[d]^{\eqref{eqn:for_!}}\\
i_! i^* \For^Q_{\bk,\bk'} \ar[r]_{\eqref{eqn:for^*}}
  & i_! \For^E_{\bk,\bk'} i^* 
}
\]
which can be checked at the level of abelian categories, as in Lemmas~\ref{lem:comp-for} and~\ref{lem:bc-for}. Finally, for~\eqref{eqn:fourier-con-vec} this is proved in Lemma \ref{lem:cohc} below.
\end{proof}

\subsection{Cohomology with compact supports and change of rings}

Consider the setting of \S\ref{ss:fourier-constant}. 

\begin{lem}
\label{lem:cohc}
There exists an isomorphism of functors
$p_! \ubb_E \cong \ubb_Y[-2\rk E]$,
see \eqref{eqn:cohc}.
Moreover, the following diagram commutes:
\[
\xymatrix@R=0.6cm{
\For^Y_{\bk,\bk'} p_! \ubb_E \ar[r]^{\eqref{eqn:for_!}} \ar[d]_{\eqref{eqn:cohc}}
  & p_! \For^E_{\bk,\bk'} \ubb_E \ar[r]^{\eqref{eqn:for-constant}}
  & p_! \ubb_E \Formod_{\bk,\bk'} \ar[d]^{\eqref{eqn:cohc}} \\
\For^Y_{\bk,\bk'} \ubb_Y [-2\rk E] \ar[rr]^{\eqref{eqn:for-constant}}
  && \ubb_Y \Formod_{\bk,\bk'} [-2\rk E] }
\]
\end{lem}

\begin{proof}
First, assume that the vector bundle is trivial: $E=\C^{\rk E} \times Y$. Then the arguments in the proof of  \cite[Proposition 3.2.3(iii)]{kas} can be easily adapted to produce our isomorphism \eqref{eqn:cohc}, and the commutativity of the diagram is obvious. If now $\phi$ is an automorphism of $E$ (as a complex vector bundle over $Y$) then $\phi$ induces an automorphism of the functor $p_! \ubb_E$. However this automorphism is trivial by the arguments of the proof of  \cite[Proposition 3.2.3(iv)]{kas}.

Now assume that $E$ can be trivialized. Then by the claim about automorphisms above, there exists a canonical isomorphism \eqref{eqn:cohc} (i.e.~any choice of a trivialization produces such an isomorphism, and it does not depend on the trivialization). The commutativity of the diagram also follows from the case $E$ is trivial.

Finally, consider the general case. Let $M$ be in $\bk\allmod$. Then the object $p_! \ub M_E[2 \rk E]$ of $\Db(Y,\bk)$ restricts to a sheaf on any open subset $Y'$ of $Y$ over which $E$ is trivializable. It follows that this object itself is a sheaf. Moreover, the restriction of this sheaf to $Y'$ is canonically isomorphic to $\ub M_{Y'}$ for any $Y'$ as above. The canonicity implies that these isomorphisms can be glued to an isomorphism $p_! \ub M_E[2 \rk E] \cong \ub M_Y$, which is clearly functorial in $M$. Finally, the commutativity of the diagram follows from the commutativity on any open subset on which $E$ is trivializable, which was checked above.
\end{proof}

\subsection{Compatibilities involving Fourier transform}

\begin{lem}
\label{lem:fourier-con-zero-section}
In the context of~\eqref{eqn:fourier-con}, the following diagram commutes, where the bottom horizontal morphism is the obvious one:
\[
\xymatrix@C=1.5cm@R=0.6cm{
y^* \bT_E \ubb_E \ar[r]^-{\eqref{eqn:zero-section}} \ar[d]_-{\eqref{eqn:fourier-con}} &  p_! \ubb_E[\rk E] \ar[d]^-{\eqref{eqn:cohc}} \\
y^* y_! \ubb_Y [-\rk E]  \ar[r] & \ubb_Y[-\rk E]
}
\]
\end{lem}
\begin{proof}
Consider the following diagram, in which the spaces and maps are defined in either \S\ref{ss:fourier-sato-defn} or~\S\ref{ss:fourier-constant}.  The commutative square on the left-hand side is cartesian.
\[
\xymatrix@R=0.5cm{
E \ar[r]_i \ar[d]_p \ar@/^2ex/[rr]^{\id} & Q \ar[r]_q \ar[d]^{\check q} & E \\
Y \ar[r]^y & E^*}
\]
Note that both $i$ and $y$ are closed inclusions, so we have adjunction morphisms $\eta: \id \to i_!i^*$ and $\epsilon: y^*y_! \to \id$.  We claim that the following diagram commutes:
\begin{equation}\label{eqn:fourier-con-zero-section-bit}
\vcenter{\xymatrix@R=0.6cm{
y^*\check q_! \ar[r]^-{\eta} \ar[d]_{\BC} &
  y^* \check q_! i_! i^* \ar[r]^-{\Co} &
  y^* y_! p_! i^* \ar[dll]^{\epsilon} \\
p_! i^*}}
\end{equation}
It suffices to check this at the level of abelian categories, and that can be done using the explicit descriptions of the various morphisms found in~\cite[Chapter~2]{kas}. 

Next, unravelling the definitions of~\eqref{eqn:fourier-con} and~\eqref{eqn:zero-section}, we find that we must prove the commutativity of the following diagram.  Here, for brevity, we put $r = -2\rk E$.
\[
\xymatrix@C=0.4cm{
y^* \check q_! q^* \ubb_E \ar[r]^-{\eta} \ar[d]_{\BC} \ar@{.>}@(dl,l)[drrr]^(.35){\eqref{eqn:zero-section}} 
  \ar@{.>}@/^7ex/[rrrr]_-{\eqref{eqn:fourier-con}}&
  y^* \check q_! i_! i^* q^* \ubb_E \ar[r]^-{\Co} &
  y^* y_! p_! i^* q^*\ubb_E \ar[r]^-{\Co} \ar@/^1ex/[dll]^(.35){\epsilon}&
  y^* y_! p_! \ubb_E \ar[r]^-{\eqref{eqn:cohc}} \ar[d]^{\epsilon}&
  y^* y_! \ubb_Y[r] \ar[d]^{\epsilon} \\
p_! i^* q^* \ubb_E \ar[rrr]_{\Co} &&&
  p_! \ubb_E \ar[r]_-{\eqref{eqn:cohc}} &
  \ubb_Y[r]}
\]
The commutativity of the upper left-hand part follows from~\eqref{eqn:fourier-con-zero-section-bit}, and the commutativity of the portion bounded by the arrows labelled `$\epsilon$' is obvious.
\end{proof}

The following two lemmas can be deduced from the basic compatibilities between $\Co$ and $\BC$ established in~\cite[Lemmas~B.4, B.6, B.7 and B.8]{ahr}.  The proofs are routine applications of the general techniques of~\cite[Appendix B]{ahr}; the details are left to the reader.

First, let $\phi : E_1 \to E_2$ be a morphism of vector bundles, and let $g : Y' \to Y$ be a continuous map. Then we can set $E_1':=Y' \times_Y E_1$, $E_2':=Y' \times_Y E_2$, and consider the natural maps
\[
f_1 : E_1' \to E_1, \ f_1' : (E_1')^* \to E_1^*, \ f_2 : E_2' \to E_2, \
f_2' : (E_2')^* \to E_2^*, \ \phi' : E_1' \to E_2'.
\]
We set $r_1:=\rk E_1$ and $r_2:=\rk E_2$.

\begin{lem}
\label{lem:fourier_!-zero-section-new}
The following diagram commutes:
\[
\xymatrix@C=1.5cm@R=0.6cm{
(\ut \phi)^* \bT_{E_1} f_{1!} [-r_1] \ar[d]_-{\eqref{eqn:fourier_!}} \ar[r]^-{\eqref{eqn:fourier^*}} 
& \bT_{E_2} \phi_! f_{1!} [-r_2] \ar[r]^-{\Co}
& \bT_{E_2} f_{2!} \phi'_![-r_2] \ar[d]^-{\eqref{eqn:fourier_!}}\\
(\ut \phi)^* f'_{1!} \bT_{E'_1}[-r_1] \ar[r]^-{\BC}
& f'_{2!} (\ut \phi ')^* \bT_{E'_1}[-r_1] \ar[r]^-{\eqref{eqn:fourier^*}}
& f'_{2!} \bT_{E_2'} \phi'_! [-r_2]
}
\] 

\end{lem}

Now, consider two morphisms of vector bundles $\phi : E_1 \to E_2$ and $\psi : E_2 \to E_3$. We set $r_1=\rk E_1$, $r_2=\rk E_2$, and $r_3= \rk E_3$.

\begin{lem}
\label{lem:fourier^*-zero-section-new}
The following diagram commutes:
\[
\xymatrix@R=0.6cm{
\ut \psi^* \ut \phi^* \bT_{E_1}[-r_1] \ar[r]^-{\Co} \ar[d]_-{\eqref{eqn:fourier^*}} &
\ut (\psi \circ \phi)^* \bT_{E_1} [-r_1] \ar[r]^-{\eqref{eqn:fourier^*}} &
  \bT_{E_3} (\psi \circ \phi)_! [-r_3] \ar[d]^-{\Co} \\
\ut \psi^* \bT_{E_2} \phi_! [-r_2] \ar[rr]^-{\eqref{eqn:fourier^*}} &&
 \bT_{E_3} \psi_! \phi_! [-r_3]
}
\]
\end{lem}

\end{document}